\documentclass[11pt]{article}
\usepackage{amsmath,amssymb,amsfonts,amsthm}
\usepackage{graphicx}
\usepackage[a4paper,margin=2.5cm]{geometry}
\usepackage{color}
\usepackage{enumitem}
\usepackage{dcolumn}
\usepackage{lscape}
\usepackage{comment}
\newcolumntype{d}[1]{D{.}{\cdot}{#1}}
\newcolumntype{.}{D{.}{.}{-1}}
\newcolumntype{,}{D{,}{,}{-1}}

\usepackage{hhline}
\usepackage{multirow}

\newcommand{\Tvrs}{T_{\mbox{\tiny{VRS-TO}}}}
\newcommand{\Efgl}{E_{\mbox{\tiny FGL}}}
\newcommand{\Tcrs}{T_{\mbox{\tiny{CRS}}}}
\newcommand{\barEfgl}{\bar{E}_{\mbox{\tiny FGL}}}
\allowdisplaybreaks

\newfont{\bg}{cmr10 scaled\magstep4} 

\newcommand{\bigzerou}{\smash{\lower1.7ex\hbox{\bg 0}}}
\usepackage{bm}
\usepackage{upgreek}
\usepackage{longtable}
\usepackage{here}
\usepackage{lineno}
\usepackage{blkarray, bigstrut}
\usepackage{mathrsfs}
\newcommand{\bbR}{\mathbb{R}}

\usepackage[%
	hidelinks,%
  setpagesize=false,%
  bookmarks=true,%
  bookmarksdepth=tocdepth,%
  bookmarksnumbered=true,%
  colorlinks=false,%
  pdftitle={},%
  pdfsubject={},%
  pdfauthor={},%
  pdfkeywords={}%
]{hyperref}
\newtheorem{theorem}{Theorem}[section]
\newtheorem{lemma}{Lemma}[section]

\definecolor{myred}{RGB}{255,50,50}         
\definecolor{mygreen}{RGB}{30,200,30}
\newcommand{\red}[1]{{#1}}

\providecommand{\keywords}[1]
{
  \small	
  \textbf{\textit{Keywords---}} #1
}

\title{A closer target setting approach to boundary problems with the Russell graph measure}
\author{
\begin{tabular}[h]{ccc}  
Atsushi Hori & Kazuyuki Sekitani \\ 
\textit{Nagoya Institute of Technology}&\textit{Seikei University}
\end{tabular}
}
\date{November 20, 2026} 
\begin{document}
\maketitle
\begin{abstract}
  A Russell graph measure (RGM) is one of the standard DEA models,
	but its efficiency measure is not well-defined---or has unacceptable properties---at the
boundary of the non-negative orthant.
This is known as a boundary problem.
  Existing studies have tackled this issue; however, their 
  models may fail to identify an efficient target or fail to  
  satisfy some desirable  properties of efficiency measures.
	In this paper, we incorporate a closer target setting approach into the RGM
  model with production trade-offs to overcome such issues.
  We demonstrate that the efficiency measure of the proposed model overcomes
	the boundary problem and has stronger properties than existing models.
	We also demonstrate that the efficiency scores of the proposed model can
  be computed by solving a series of LPs. 
  We conduct a numerical experiment with a real-world dataset
  to illustrate how targets provided by our model are realistic
  compared with the existing model, which also suggests the
  validity of our model in applications.
\end{abstract}
\section{Introduction}
Data envelopment analysis (DEA)~\cite{banker1984some,charnes1978measuring}
is an effective  nonparametric method for estimating the set of feasible combinations of inputs and outputs, also called the production possibility set (PPS), and 
for evaluating the efficiency of decision making units (DMUs) relative to the estimated set. 
DEA models are formulated as mathematical programming problems of maximization or minimization, which are flexibly defined by 
the choice of PPS and objective function depending on DEA applications. 
One of the practical advantages of DEA is to provide an efficiency score and
target for DMUs, which are often computed via linear programming (LP). 
\par
\par
Efficiency scoring of an input--output vector of the PPS  is a map from
the PPS to the range $[0,1]$, which is referred to as an efficiency measure.  
Our research is situated in the field of axiomatic analysis in the literature
on efficiency measurement.
F{\"a}re and Lovell~\cite{fare1978measuring} provided three fundamental axioms of 
the input-oriented efficiency measure as follows:
\begin{itemize}
\item input-indication (the measure is equal to one if and only if the input vector is technically efficient in the sense of~\cite{koopmans1951analysis});
\item input-homogeneity (e.g., doubling all input quantities while holding all outputs constant cuts the measure in half);
\item strong monotonicity in inputs (increasing one input quantity while holding all other inputs and all outputs constant lowers the measure). 
\end{itemize}
By virtue of Russell~\cite{10.1007/978-3-642-52481-3_18}, Blackorby and Russell~\cite{blackorby1999aggregation}, and Russell and Schworm~\cite{10.2307/23883832},  
the initial axioms~\cite{fare1978measuring} were clarified and extended to axioms of  an efficiency measure defined on the full space of positive inputs and positive outputs.
Specifically, Russell and Schworm~\cite{10.2307/23883832} showed that no inefficiency measure can satisfy both indication and continuity.
This incompatibility shows that any efficiency measure of DEA models
on the positive full space is classified into two groups: 
Those that are continuous in the PPS and those that satisfy indication  since any efficiency measure  is equivalently reduced to an inefficiency measure. 
In fact, the latter group includes representative efficiency measures of popular DEA models such as the Russell graph measure (RGM,~\cite{fare1985measurement}), the slacks-based measure
(SBM,~\cite{tone2001slacks}) which is equivalent to the
enhanced RGM~\cite{pastor1999enhanced} and the Range adjusted measure~\cite{cooper1999ram}; see also Sueyoshi and Sekitani~\cite{SUEYOSHI2009764}. 
\par
The RGM~\cite{fare1985measurement} 
simultaneously accounts for the inefficiency in both inputs and outputs.
The RGM  satisfies unit invariance, but it fails to satisfy both 
indication and strong monotonicity in the full space of  non-negative inputs and  non-negative outputs.
Levkoff et al.~\cite{levkoff2012boundary} pointed out that the failure  in the  non-negative full space
is caused by the adjustment work for  zero output data, which is called the boundary problem.
To address the problem, they slightly modified the objective function of the RGM DEA model.
Their efficiency measure  satisfies indication and weak monotonicity,
which is a relaxation of strong monotonicity, but the modified RGM model 
may have no optimal solution, thereby failing to provide any target for the assessed DMU.
\par
This study aims to  overcome the boundary problem by  incorporating a closer target setting approach into the RGM model with production trade-offs. 
The closer target setting approach 
assumes  that closest targets are very similar to the assessed DMU, and the closest targets also lead to efficiency projections that may be reached with less effort than other alternatives.
Recent developments of the closer target setting DEA models are summarized  in~\cite{aparicio2016survey,aparicio2007closest,portela2003finding,razipour2020finding,Sekitani02032025}.  
To  enhance the reliability of the RGM and the  practicality of the targets,
we  show that the proposed  RGM satisfies  strong monotonicity and its  target achieves 
the least input-distance or the least output-distance.
Moreover, the proposed RGM has a computational advantage such that  the  efficiency measurement is implemented by solving a series of LPs.
\par
As a relevant study,
Sekitani and Zhao~\cite{Sekitani02032025} also developed  the closer target setting approach to the RGM model under an empirical  production possibility without a trade-offs axiom.
Furthermore, they restricted the assessed DMUs to those with positive input and output data. 
Hence, the closer target setting approach to RGM by~\cite{Sekitani02032025} cannot overcome 
the boundary problem, while our approach can.
\par
The paper unfolds as follows: 
Section~\ref{sec:prelim} introduces the notation and describes the 
general assumptions on the production possibility set and 
desirable properties of efficiency measures.
Section~\ref{sec:model} 
discusses existing RGM-type efficiency measures on trade-offs axioms 
in order to clarify the motivation of this study.
Section~\ref{sec:effMeasure} introduces the extension of the 
closer target setting approach to RGM and proves that it satisfies 
indication and strong monotonicity at the boundary. 
Section~\ref{sec:Comp} discusses the implementation of checking frontier assumptions
and detection of the so-called \emph{free lunch} which produces outputs with zero input.
Section~\ref{sec:examples} illustrates the
practicality of the concept by applying it to a real-world dataset,
on Olympic Games performance. Section~\ref{sec:conclusion} concludes
this paper.

\section{Preliminaries}\label{sec:prelim}
Consider $n$ DMUs, and let DMU$_j$ be denoted as the $j$th DMU,
$j\in\{1,\dots,n\}$.
For each DMU, there are $m$ inputs $\bm{x}_j:=(x_{1j},\ldots,x_{mj})^\top\in\bbR^m$ 
and $s$ outputs $\bm{y}_j:=(y_{1j},\ldots,y_{sj})^\top\in\bbR^s$,
where $^\top$ denotes the transpose of a vector.
Unless otherwise noted, the concatenated vector is denoted without $^\top$, e.g.,
$(\bm{x},\bm{y}) = (\bm{x}^\top,\bm{y}^\top)^\top\in\bbR^{m+s}$ for $\bm{x}\in\bbR^m$ and $\bm{y}\in\bbR^s$.
Let $\bbR^D_+:=\{\bm{z}\in\bbR^D\mid \bm{z}\ge\bm{0}\}$,
$\bbR^D_{++}:=\{\bm{z}\in\bbR^D\mid \bm{z}>\bm{0}\}$,
$\bar{\bbR}^D_{+}:=\bbR^D_+\setminus \{\bm{0}\}$,
and $\bm{1}_D$ be the all-ones $D$-dimensional vector.
Assume that
$(\bm{x}_j,\bm{y}_j)\in  \bar{\bbR}^{m}_{+}  \times \bar{ \bbR}^{s}_{+}$
for all $j=1,\ldots,n$. 
\par
According to Podinovski~\cite{podinovski2004production}, the trade-offs axiom 
for a PPS $T$ and directions $(\bm{r}^-_t,\bm{r}^+_t)\in \bbR^m\times \bbR^s$ $(t=1,\ldots,K)$  is defined as follows:
$(\bm{x}+\sum_{t=1}^K \pi_t \bm{r}^-_t, \bm{y}+\sum_{t=1}^K \pi_t \bm{r}^+_t )\in T$ for any  $\bm{\pi}\geq \bm{0}$
and any  $(\bm{x},\bm{y})\in T$ satisfying $(\bm{x}+\sum_{t=1}^K \pi_t \bm{r}^-_t, \bm{y}+\sum_{t=1}^K \pi_t \bm{r}^+_t )\in \bbR^m_+\times \bbR^s_+$.
Let $R^-:=[\bm{r}^-_1\cdots \bm{r}^-_K]\in\bbR^{m\times K}$ and $R^+:=[\bm{r}^+_1\cdots \bm{r}^+_K]\in\bbR^{s\times K}$.
\red{A direction $(\bm{r}^-_t,\bm{r}^+_t)$ represents a substitution relationship between inputs and outputs, among inputs, or among outputs~\cite{podinovski2004production},
and this relationship is assumed to hold throughout a PPS.
For example, in the efficiency analysis of nations that won Olympic medals,
the following substitution relationship holds for every nation:
``if a nation has at least one silver medal, reducing its silver medal count by one
is assumed to allow its bronze medal count to increase by one instead.''}
\par
Define the PPS $\Tvrs$ with production trade-offs as follows:
\begin{align}
\Tvrs:=
\left\{ 
\left(\bm{x},\bm{y}\right) 
\begin{array}{|l}
 \sum_{j=1}^n \lambda_j \bm{x}_j +\sum_{t=1}^K \pi_t \bm{r}^-_t \leq \bm{x}, \\
 \sum_{j=1}^n \lambda_j \bm{y}_j +\sum_{t=1}^K \pi_t \bm{r}^+_t \geq \bm{y}, \\
 \sum_{j=1}^n \lambda_j =1, \\
 \lambda_j \geq 0\ (j=1,\ldots,n), \\
  \pi_t \geq 0\ (t=1,\ldots,K)  
\end{array}
\right\}
\cap \left( \bbR^{m}_+\times  \bar{\bbR}^{s}_{+} \right).   \label{PPS}
\end{align}
If $(\bm{r}^-_t,\bm{r}^+_t)=(\bm{0},\bm{0})$ for all $t=1,\ldots,K$, then $\Tvrs$ has no production trade-offs between inputs and outputs, which reduces to a conventional PPS under the variable returns to scale~\cite{banker1984some}.
Let $(\bm{r}^-_t, \bm{r}^+_t) := (\bm{x}_t, \bm{y}_t)$  for each $t = 1,\ldots, K$,
where $K=n$, and $(\bm{x}_{n+1}, \bm{y}_{n+1}) := (\bm{0}, \bm{0})$.
Then, $\Tvrs$ is reduced to a classic PPS under the constant returns to scale developed by  Charnes et al.~\cite{charnes1978measuring}. This is denoted by $T_{\mbox{\tiny CRS}}$.
\par
Tone~\cite{tone2001slacks} developed the so-called SBM-AR model which uses the following superset of $\Tvrs$:
\begin{align}
P := 
\left\{ 
\left(\bm{x},\bm{y}\right) 
\begin{array}{|l}
 \sum_{j=1}^n \lambda_j \bm{x}_j +\sum_{t=1}^K \pi_t \bm{r}^-_t \leq \bm{x}, \\
 \sum_{j=1}^n \lambda_j \bm{y}_j +\sum_{t=1}^K \pi_t \bm{r}^+_t \geq \bm{y}, \\
 \sum_{j=1}^n \lambda_j =1, \\
 \lambda_j \geq 0\ (j=1,\ldots,n), \\
  \pi_t \geq 0\ (t=1,\ldots,K)  
\end{array}
\right\}.\label{eqP}
\end{align}
The set $P$ is 
a polyhedron which has a finite number of facets, and hence there exist
$(\bm{v}^l,\bm{u}^l)\in \bar{\bbR}^{m+s}_+$ and a scalar $\sigma^l\in \bbR$
for all $l=1,\ldots,L$ such that 
\begin{align}
P=
  \{(\bm{x},\bm{y})\mid\bm{v}^l\bm{x}-\bm{u}^l\bm{y}-\sigma^l \geq 0,
  \ l=1,\dots,L\}, \label{P}
\end{align}
where all pairs $(\bm{v}^l,\bm{u}^l)$, $l=1,\dots,L$, 
 denote each normal direction for  facets  of  $P$ and $\sigma^l=\min\{ \bm{v}^l\bm{x}-\bm{u}^l\bm{y} \, \left|\, (\bm{x},\bm{y})\in P \right. \}.$
Hereinafter, suppose that $\bm{v}^l$ and $\bm{u}^l$ are row vectors.
The PPS $\Tvrs$ defined by~\eqref{PPS} is written as
\[
T_{\mbox{\tiny VRS-TO}}=P\cap\left(   \bbR^{m}_+\times \bar{\bbR}^{s}_{+} \right)=
  \{(\bm{x},\bm{y})\mid\bm{v}^l\bm{x}-\bm{u}^l\bm{y}-\sigma^l \geq 0,
  \ l=1,\dots,L\} \cap \left( \bbR^{m}_+\times  { \bar{\bbR}^{s}_{+}} \right).
\]

Let $\cal T$ be the class of PPS for the production trade-offs technology with $K$ feasible directions $\{ (\bm{r}^-_1,\bm{r}^+_1),\ldots, (\bm{r}^-_K,\bm{r}^+_K)\}$. 
For a PPS $T\in {\cal T}$, we define the strongly and weakly efficient 
frontiers as follows:
\begin{align}
\partial^s({T}):= \left\{\, (\bm{x},\bm{y})\in {T} \, \left|
\begin{array}{l}  (\bm{x},-\bm{y}) \geq  (\bm{x}',-\bm{y}'),  \\
(\bm{x},-\bm{y}) \not=  (\bm{x}',-\bm{y}') 
\end{array}
\implies   (\bm{x}',\bm{y}')\notin {T}
\right\} \right.,\label{partialS} \\
\intertext{and}
\partial^w({T}):= \left\{\, (\bm{x},\bm{y})\in {T} \, \left|
 (\bm{x},-\bm{y}) >  (\bm{x}',-\bm{y}') 
\implies   (\bm{x}',\bm{y}')\notin {T}
\right\}, \right. \label{partialW} 
\end{align}
respectively.

The following lemmas play central roles in this paper.
\begin{lemma}
  \label{lem:positive.vandu}
	For any \emph{PPS} $P$ given by~{\textup{\red{\eqref{eqP}}}},
  $\partial^s(P)=\partial^w(P)$ if and only if  
  $(\bm{v}^l,\bm{u}^l)\in \bbR^{m+s}_{++}$ for all $l=1,\ldots,L$.
\end{lemma}
\begin{proof}
	Since the set $P$ has $L$ facets, say $F^1,\ldots,F^L$, we have $\partial^w(P)=\cup_{l=1}^L F^l$ and $F^l\not=\emptyset$ for all $l=1,\ldots,L$.
It follows from the normal direction $(\bm{v}^l,\bm{u}^l)$ of the facet  $F^l\subseteq \partial^w(P)$ that 
$
F^l = \left\{ (\bm{x},\bm{y})\in P \,|\, \bm{v}^l\bm{x}-\bm{u}^l\bm{y}-\sigma^l=0 \right\}.
$
\par 
For any $(\bm{x},\bm{y})\in \partial^w(P)$,
there exists a facet $F^{l'}$ satisfying   $(\bm{x},\bm{y})\in F^{l'}$. 
For all $l=1,\ldots,L$,   $(\bm{v}^l,\bm{u}^l)\in \bbR^{m+s}_{++}$ 
if and only if  
\begin{align}
	0<&\min\left\{ \bm{v}^l\bm{\epsilon}^-+\bm{u}^l\bm{\epsilon}^+ |\, 
		\bm{1}_m^\top\bm{\epsilon}^-+\bm{1}_s^\top\bm{\epsilon}^+=1, \bm{\epsilon}^-\geq \bm{0},  \bm{\epsilon}^+\geq \bm{0} 
\right\} \label{ieq:positiveVU}\\
	=&\max\left\{ \delta |\, \delta \bm{1}_m^\top \leq \bm{v}^l,\, \delta \bm{1}_s^\top \leq \bm{u}^l \right\},\notag
\end{align}
where the last equality follows from strong duality.
Therefore,   it follows from
$(\bm{x},\bm{y})\in F^{l'}$, $(\bm{v}^{l'},\bm{u}^{l'})\in \bbR^{m+s}_{++}$,
and~\eqref{ieq:positiveVU} that
\begin{align*}
\bm{v}^{l'}(\bm{x}-\bm{\epsilon}^-)-\bm{u}^{l'}(\bm{y}+\bm{\epsilon}^+)-\sigma^{l'}=0 -(\bm{v}^{l'}\bm{\epsilon}^-+\bm{u}^{l'}\bm{\epsilon}^+)<0
\mbox{  for any } (\bm{\epsilon}^-,\bm{\epsilon}^+)\in \bar{\bbR}^{m+s}_{+}.
\end{align*}
That is,
$(\bm{x}-\bm{\epsilon}^-,\bm{y}+\bm{\epsilon}^+)\not\in P$  
for any $(\bm{\epsilon}^-,\bm{\epsilon}^+)\in \bar{\bbR}^{m+s}_{+}$.
Therefore, we have  $(\bm{x},\bm{y})\in \partial^s(P)$, and this leads to $\partial^w(P)\subseteq \partial^s(P)$.
Since~\eqref{partialS} and~\eqref{partialW} imply $\partial^s(P) \subseteq \partial^w(P)$, 
we have $\partial^s(P) = \partial^w(P)$ if
$(\bm{v}^l,\bm{u}^l)\in \bbR^{m+s}_{++}$ for all $l=1,\ldots,L$.
\par
\leavevmode\red{Conversely, choose  $(\bm{x}^l,\bm{y}^l)$ from the relative interior of
$F^l$ for all $l=1,\ldots,L$.
Since $(\bm{x}^l,\bm{y}^l) \in \partial^w(P) = \partial^s(P)$, $\bm{v}^{{l}}\bm{x}^l-\bm{u}^{{l}}\bm{y}^l-\sigma^{{l}}=0$,  and
$\bm{v}^{\bar{l}}\bm{x}^l-\bm{u}^{\bar{l}}\bm{y}^l-\sigma^{\bar{l}}>0$ for all $\bar{l}\not=l$,
there exists $(\bar{\bm{\epsilon}}^-,\bar{\bm{\epsilon}}^+)\in {\bbR}^{m+s}_{++}$ such that 
$\bm{v}^{\bar{l}}(\bm{x}^l-\bm{\epsilon}^-)-\bm{u}^{\bar{l}}(\bm{y}^l+\bm{\epsilon}^+)-\sigma^{\bar{l}} \geq 0$
for any  non-negative vector $(\bm{\epsilon}^-,\bm{\epsilon}^+)\leq (\bar{\bm{\epsilon}}^-,\bar{\bm{\epsilon}}^+)$
and  all $\bar{l}\not=l$.
Since $(\bm{x}^l,\bm{y}^l)\in\partial^s(P)$, we have
\begin{align*}
\bm{v}^{l}(\bm{x}^l-\bm{\epsilon}^-)
-\bm{u}^{l}(\bm{y}^l+\bm{\epsilon}^+)-\sigma^{l}
=0-(\bm{v}^{l}\bm{\epsilon}^-+\bm{u}^{l}\bm{\epsilon}^+)<0
\end{align*}
for every $(\bm{\epsilon}^-,\bm{\epsilon}^+)\in
\bar{\bbR}^{m+s}_{+}$ satisfying
$(\bm{\epsilon}^-,\bm{\epsilon}^+)
\leq(\bar{\bm{\epsilon}}^-,\bar{\bm{\epsilon}}^+)$.
Applying this inequality to perturbations with exactly one positive
component yields $(\bm{v}^l,\bm{u}^l)\in\bbR^{m+s}_{++}$.}
\end{proof}


\begin{lemma}
  \label{lem:hyperplane}
  For any {\rm PPS} $\Tvrs\in{\cal T}$, 
	\red{let $P$ be the polyhedron defined by~\eqref{eqP} and}
  assume that $\partial^s(P)=\partial^w(P)$.
  Then, the following conditions are valid:
  \begin{enumerate}[label={\rm (\alph*)}]
    \item \label{hyper.a}
    $\left(\bm{x},\bm{y}\right)\in
    \partial^s\left(T_{\mbox{\tiny VRS-TO}}\right)$
    if and only if there exists $\bar{l}\in\{1,\dots,L\}$
    such that
    $\bm{v}^{\bar{l}}\bm{x}-\bm{u}^{\bar{l}}
    \bm{y}-$ $\sigma^{\bar{l}}=0;
    $
    \item \label{hyper.b}
    $\left(\bm{x},\bm{y}\right)\in
      \partial^w\left(T_{\mbox{\tiny VRS-TO}}\right)
      \setminus  \partial^s\left(T_{\mbox{\tiny VRS-TO}}\right)$
      if and only if
      \[
      \begin{cases}
        x_{\bar{i}}=0\quad \exists \bar{i}\in \left\{1,\ldots,m \right\};\\
        \bm{v}^l\bm{x}-\bm{u}^l \bm{y}-\sigma^l>0\quad \forall
        l\in\{1,\ldots,L\}.
      \end{cases}
      \]
  \end{enumerate}
\end{lemma}
\begin{proof}
	Since $(\bm{x},\bm{y})\in P\cap(\bbR^m_+\times \bar{\bbR}^s_+)$ is equivalent to
	$(\bm{x},\bm{y})\in \Tvrs$, it follows from the duality theorem of LP that 
\begin{align}
 & \max\left\{\left.
		 \bm{1}_m^\top \bm{\epsilon}^- + \frac{1}{s}\bm{1}_s^\top\bm{\epsilon}^+\, \right|\, (\bm{x}-\bm{\epsilon}^-,\bm{y}+\bm{\epsilon}^+) \in P, (\bm{\epsilon^-},\bm{\epsilon}^+)\in \bbR^{m+s}_+  
\right\} \label{eqTop} \\
 &\geq \max\left\{\left.
		 \bm{1}_m^\top \bm{\epsilon}^- + \frac{1}{s} \bm{1}_s^\top \bm{\epsilon}^+\, \right|\, (\bm{x}-\bm{\epsilon}^-,\bm{y}+\bm{\epsilon}^+) \in \Tvrs, (\bm{\epsilon^-},\bm{\epsilon}^+)\in \bbR^{m+s}_+  
\right\} \label{eqSecond} \\
 & \geq  \max\left\{\left. \frac{1}{s} \bm{1}_s^\top \bm{\epsilon}^+\, \right| (\bm{x}-\bm{\epsilon}^-,\bm{y}+\bm{\epsilon}^+) \in \Tvrs, (\bm{\epsilon^-},\bm{\epsilon}^+)\in \bbR^{m+s}_+ 
\right\} \notag \\
&=\min\left\{(\bm{v}+\bm{w})\bm{x}-\bm{u}\bm{y}-\sigma   \left| 
\begin{array}{l}
\bm{v}\bm{x}_j -\bm{u}\bm{y}_j -\sigma \geq 0 \, (j=1,\ldots,n), \bm{v}R^--\bm{u}R^+ \geq \bm{0}, \\
\bm{v}+\bm{w}\geq \bm{0},\\
\bm{u} \geq \frac{1}{s} \bm{1}_s^\top, \bm{v}\geq \bm{0}, \bm{w}\geq \bm{0}, \bm{u} \geq \bm{0}
\end{array}
\right. 
\right\}\notag\\
&\geq \min\left\{\bm{v}\bm{x}-\bm{u}\bm{y}-\sigma   \left| 
\begin{array}{l}
\bm{v}\bm{x}_j -\bm{u}\bm{y}_j -\sigma \geq 0 \, (j=1,\ldots,n), \bm{v}R^--\bm{u}R^+ \geq \bm{0}, \\
\bm{u} \geq\frac{1}{s}  \bm{1}_s^\top,\bm{v}\geq \bm{0}, \bm{u} \geq \bm{0}
\end{array}
\right. 
\right\} \notag\\ 
&\geq \min\left\{\bm{v}\bm{x}-\bm{u}\bm{y}-\sigma   \left| 
\begin{array}{l}
\bm{v}\bm{x}_j -\bm{u}\bm{y}_j -\sigma \geq 0 \, (j=1,\ldots,n),\bm{v}R^--\bm{u}R^+ \geq \bm{0}, \\
\bm{v}\bm{1}_m+ \bm{u} \bm{1}_s\geq 1, \bm{v}\geq \bm{0}, \bm{u} \geq \bm{0}
\end{array}
\right. 
\right\}\label{eq3rd}\\
&= \max\left\{ \delta  \left| (\bm{x}-\delta \bm{1}_m,\bm{y}+\delta\bm{1}_s) \in P,\ \delta \geq 0 \right. 
\right\}\geq 0. \label{eqBottom} 
\end{align}
\par
For any  $(\bm{x},\bm{y})\in \partial^s(P)$ with $(\bm{x},\bm{y})\in \bbR^m_+\times \bar{\bbR}^s_+$, both~\eqref{eqTop} and~\eqref{eqSecond}  have the optimal value $0$.
Hence, we have 
 $ \left\{ (\bm{x},\bm{y})\,
\left| \,
(\bm{x},\bm{y})\in \bbR^m_+\times \bar{\bbR}^s_+,\ (\bm{x},\bm{y})\in \partial^s(P)
\right.
\right\} \subseteq \partial^s(\Tvrs).$
\par
Assume  $\partial^w(P)=\partial^s(P)$, and let $(\bm{x},\bm{y})\in \partial^s(\Tvrs)$ be arbitrary. Then,  both~\eqref{eqSecond} and~\eqref{eqBottom} have the optimal value $0$.  
Hence,   $\partial^s(\Tvrs) \subseteq \partial^w(P)\cap (\bbR^m_+\times \bar{\bbR}^s_+)
=  \partial^s(P)\cap (\bbR^m_+\times \bar{\bbR}^s_+). $ 
Therefore, we have $\partial^s(\Tvrs) = \partial^s(P)\cap (\bbR^m_+\times \bar{\bbR}^s_+)$, 
which, together with $\partial^w(P)=\cup_{l=1}^L \{(\bm{x},\bm{y})\in P \,|\, \bm{v}^l\bm{x}-\bm{u}^l\bm{y}-\sigma^l=0 \}$, implies that
\ref{hyper.a} is valid.
\par
Moreover, $(\bm{x},\bm{y})\in  \partial^w(\Tvrs) \setminus \partial^s(\Tvrs)$ if and only if 
\begin{align}
& (\bm{x},\bm{y})\in  \partial^w(\Tvrs) \setminus \partial^s(P) \label{eqConb} \\
\iff  & (\bm{x},\bm{y})\in  \partial^w(\Tvrs) \setminus \partial^w(P)\notag \\
\iff&
 0= \max\left\{ \delta  \left| (\bm{x}-\delta \bm{1}_m,\bm{y}+\delta\bm{1}_s) \in \Tvrs, \, \delta\geq 0 \right. 
\right\} \label{eqPWT.dual}\\
&\mbox{ and  the optimal value of~\eqref{eqBottom}  is positive. }\notag\\
\iff& 0=\min\left\{(\bm{v}+\bm{w})\bm{x}-\bm{u}\bm{y}-\sigma   \left| 
\begin{array}{l}
\bm{v}\bm{x}_j -\bm{u}\bm{y}_j -\sigma \geq 0 \,
 (j=1,\ldots,n),\\
\bm{v}R^--\bm{u}R^+ \geq \bm{0},\, 
\bm{v}\bm{1}_m+\bm{w}\bm{1}_m +  \bm{u}\bm{1}_s\geq  1\\
\bm{v}\geq \bm{0}, \bm{w}\geq \bm{0}, \bm{u} \geq \bm{0}
\end{array}
\right. 
\right\}\label{eqPWT} \\
&\mbox{ and  the optimal value of~\eqref{eq3rd}  is positive, }\notag
\end{align}
where  the minimization problem~\eqref{eqPWT} is  the dual problem of~\eqref{eqPWT.dual}.
For any $(\bm{v},\bm{u},\sigma)$ satisfying   $
\bm{v}\bm{x}_j -\bm{u}\bm{y}_j -\sigma \geq 0$ for $j=1,\ldots,n$, 
$\bm{v}R^--\bm{u}R^+ \geq \bm{0}$, and $(\bm{v}, \bm{u})\in \bar{\bbR}^{m+s}_+$, 
it follows  from the positive optimal value of~\eqref{eq3rd} that $\bm{v}\bm{x}-\bm{u}\bm{y}-\sigma>0$. 
The zero optimal value of \eqref{eqPWT} implies that an optimal solution $(\bm{v}^*,\bm{u}^*,\bm{w}^*,\sigma^*)$ to~\eqref{eqPWT}  satisfies $(\bm{v}^*,\bm{u}^*)=(\bm{0},\bm{0})$ and $\sigma^*=0$.
Moreover, we have $\bm{w}^*\bm{x}=0$ and $\bm{w}^*\in \bar{\bbR}^m_+$; that is, there exists $\bar{i}\in \{1,\ldots,m\}$ such that $w^*_{\bar{i}}>0$ and $x_{\bar{i}}=0$.  
Combining this observation with~\eqref{eqConb} establishes~\ref{hyper.b}.
\end{proof}

The efficiency measure is a mapping $F:T\times {\cal T} \to [0,1]$.
Consider the following three types of axioms for efficiency measures
defined on the full space of inputs and outputs.
These are extensions of the axioms proposed by
\cite{fare1978measuring} for input-oriented measures of 
efficiency, and we consider a weaker concept of monotonicity as well.
\begin{description}
\item{\bf Indication of Efficiency (I):}
	For a PPS $T\in {\cal T}$ and  all $(\bm{x},\bm{y})\in {T}$, $F(\bm{x},\bm{y}; {T})=1$ if and only if $(\bm{x},\bm{y})\in \partial^s({T})$. 
\item{\bf Strong monotonicity (SM):}
	For a PPS $T\in {\cal T}$ and all pairs  $(\bm{x},\bm{y})\in {T}$ and  $(\bm{x}',\bm{y}')\in {T}$ satisfying  $(\bm{x},-\bm{y}) \leq  (\bm{x}',-\bm{y}')$ and 
$(\bm{x},-\bm{y}) \not=  (\bm{x}',-\bm{y}')$, $F(\bm{x},\bm{y};{T}) > F(\bm{x}',\bm{y}';{T}).$
\item{\bf Weak monotonicity (WM):}
	For a PPS $T\in {\cal T}$ and all pairs  $(\bm{x},\bm{y})\in {T}$ and  $(\bm{x}',\bm{y}')\in {T}$ satisfying  $(\bm{x},-\bm{y}) \leq  (\bm{x}',-\bm{y}')$, $F (\bm{x},\bm{y};{T}) \geq F(\bm{x}',\bm{y}';{T}).$   
\end{description}
\section{Russell Graph Measure DEA Model}\label{sec:model}

In this section, we review existing studies on the RGM and clarify the motivation for our proposed model.


The Russell graph measure (RGM)~\cite{fare1985measurement} for the assessed input--output vector $(\bm{x},\bm{y})\in T_{\mbox{\tiny VRS-TO}}\cap \left( \bbR^{m}_{++}\times  \bbR^{s}_{++} \right)
$ is given by the optimal value of
\begin{eqnarray}
 \min        &&  f(\bm{\theta},\bm{\phi}):=
 \frac{1}{m+s}\left( \sum_{i=1}^m \theta_i + \sum_{r=1}^s \frac{1}{\phi_{r}} \right) \label{P0}\\
 \mbox{s.t.} &&  \sum_{j=1}^n \lambda_j {x}_{ij} +\sum_{t=1}^K  \pi_{t} r^-_{it}  \leq  \theta_i {x_i},  \ i=1,\ldots,m,  \label{P1}\\
  &&  \sum_{j=1}^n \lambda_j {y}_{rj} + \sum_{t=1}^K \pi_t {r}^+_{rt} \geq  \phi_r {y}_r,  \ r=1,\ldots,s, \label{P2}\\
  && \sum_{j=1}^n \lambda_j =1, \label{P3} \\
  && \lambda_j \geq 0,\ j=1,\ldots,n, \  \pi_t \geq 0,\ t=1,\ldots,K, \label{P4}  \\
  && 0\leq \theta_i\leq 1, \ i=1,\ldots,m, \ 1 \leq \phi_r, \ r=1,\ldots,s. \label{P5}        
\end{eqnarray} 
Denote by $H(\bm{x},\bm{y}; T_{\mbox{\tiny VRS-TO}})$ the optimal value of the minimization problem~\eqref{P0}--\eqref{P5}.
Let $\bm{\theta}\otimes\bm{x}$ and $\bm{\phi}\otimes\bm{y}$ denote $(\theta_1x_1,\ldots,\theta_mx_m)$ and
$(\phi_1y_1,\ldots,\phi_sy_s)$, respectively.
Then, the minimization problem~\eqref{P0}--\eqref{P5} is formulated as 
\begin{eqnarray}
 \min        &&  f(\bm{\theta},\bm{\phi})  \label{P6}\\
 \mbox{s.t.} && (\bm{\theta}\otimes\bm{x},\bm{\phi}\otimes\bm{y}) \in \Tvrs,~\eqref{P5}\red{.} \label{P8}
\end{eqnarray}   
Let $(\bm{\theta}^*,\bm{\phi}^*)$ be an optimal solution to~\eqref{P6}--\eqref{P8}.
Then, we have
\begin{align}
  (\bm{\theta}^*\otimes\bm{x},\bm{\phi}^*\otimes\bm{y}) \in \partial^s(T_{\mbox{\tiny VRS-TO}})
\end{align}
since  $f(\bm{\theta}',\bm{\phi}')<f(\bm{\theta},\bm{\phi})$ for all
$(\bm{\theta}',-\bm{\phi}')\le (\bm{\theta},-\bm{\phi})$ with 
$(\bm{\theta}',-\bm{\phi}') \not=  (\bm{\theta},-\bm{\phi})$,
and  for any  $(\bm{\theta}\otimes\bm{x},\bm{\phi}\otimes\bm{y}) \in T_{\mbox{\tiny VRS-TO}} \setminus \partial^s(T_{\mbox{\tiny VRS-TO}})$,
there exists $(\bm{\theta}'\otimes\bm{x},\bm{\phi}'\otimes\bm{y}) \in \partial^s(T_{\mbox{\tiny VRS-TO}})$ satisfying
$(\bm{\theta}'\otimes\bm{x},-\bm{\phi}'\otimes\bm{y}) \leq (\bm{\theta}\otimes\bm{x},-\bm{\phi}\otimes\bm{y})$.
Therefore, the minimization problem~\eqref{P6}--\eqref{P8} is equivalently reduced to 
 \begin{align}
 \min\left\{\,  f(\bm{\theta},\bm{\phi}) \,\left|\,
  (\bm{\theta}\otimes\bm{x},\bm{\phi}\otimes\bm{y}) \in \partial^s(T_{\mbox{\tiny VRS-TO}}), \eqref{P5}
 \right\}.\right.\label{minRM}
 \end{align}

As stated by~\cite{levkoff2012boundary,Sekitani02032025}, the minimization problem~\eqref{P0}--\eqref{P5} has
boundary problems that occur in observed output data containing zero $y_r=0$ and target input data with zero $\theta_i^*x_i=0$.
When $y_r=0$ for some $r \in \{1,\ldots,s\}$,
Levkoff et al.~\cite{levkoff2012boundary} showed that the minimization problem~\eqref{P0}--\eqref{P5} has  no optimal solution, and hence
$H(\bm{x},\bm{y}; T_{\mbox{\tiny VRS-TO}})$ is not well-defined.
Moreover,  Levkoff et al.~\cite[Theorem~1]{levkoff2012boundary} showed that the modification of  $H(\cdot,\cdot; T_{\mbox{\tiny VRS-TO}})$ for $y_r=0$  fails to satisfy (I) and (WM).  
\red{Sekitani and Zhao~\cite{Sekitani02032025} showed that if an optimal solution at $(\bm{x},\bm{y}) \in \Tvrs\cap \bbR^{m+s}_{++}$
	satisfies $\theta_i^*x_i=0$ with $x_i>0$ for some $i$,
then $H(\cdot,\cdot; T_{\mbox{\tiny VRS-TO}})$ fails to satisfy (SM).}
Moreover, Sekitani and Zhao~\cite{Sekitani02032025} reported from 
experiments with real-world DEA applications that the minimization 
problem~\eqref{P0}--\eqref{P5} often provides the optimal input vector
$\bm{\theta}^*\otimes \bm{x}=\bm{0}$  even if 
$(\bm{x},\bm{y})\in \bbR^{m+s}_{++}$, which is referred to as the
\emph{free lunch issue}.
\par
          \leavevmode\red{Generally speaking, for a PPS $T$ a vector $(\bm 0,\bm y)\in T$ with $\bm y\in\bar{\bbR}^s_+$ is
	called a \emph{free lunch vector}.
	Loosely, free lunch means that output can be produced with no input.
	We say that $T$ \emph{allows free lunch} if it contains such a vector;
	otherwise, $T$ satisfies the \emph{no free lunch axiom}.}

Hence, the optimal input--output vector $(\bm{\theta}^*\otimes \bm{x}, \bm{\phi}^*\otimes \bm{y})\in T_{\mbox{\tiny VRS-TO}}$ with  $\bm{\theta}^*\otimes \bm{x}=\bm{0}$  and $\bm{\phi}^*\otimes \bm{y} \in 
\bar{\bbR}^{s}_{+}$ 
contradicts the no free lunch production axiom.  
The minimization problem~\eqref{P0}--\eqref{P5} may provide an unrealistic target for  $(\bm{x},\bm{y})\in T_{\mbox{\tiny VRS-TO}}\cap \bbR^{m+s}_{++}$.
\par
\leavevmode\red{An axiomatic approach to efficiency measurement~\cite{dmitruk1991existence,fare1983measuring,10.1007/978-3-642-52481-3_18} assumes  the no free lunch axiom on a PPS. However, 
PPSs used in DEA case studies  may allow free lunch production~\cite{allen1997weights,Sekitani02032025}.
The free lunch vector should be removed from the PPS. However, the resulting PPS after the removal should be maximized. 
Moreover, the analysis results should remain unchanged after the removal of free lunch vectors from the PPS.
One possible way to remove the free lunch vector  is to use the set
$P\cap(\bar{\bbR}^m_+\times\bbR^s_+)$ instead of $T_{\mbox{\tiny VRS-TO}}$. Using the set $P\cap(\bar{\bbR}^m_+\times\bbR^s_+)$ removes only the free lunch vector.}
\par
\leavevmode\red{Suppose that there exists a DEA model with the PPS $T_{\mbox{\tiny VRS-TO}}$ such that every target has a nonzero input vector for any $(\bm{x},\bm{y})\in  T_{\mbox{\tiny VRS-TO}}\cap (\bar{\bbR}^{m}_{+}\times \bbR^{s}_{+})$. 
Then, its efficiency score is invariant to changing the PPS from $T_{\mbox{\tiny VRS-TO}}$  to $P\cap(\bar{\bbR}^m_+\times\bbR^s_+)$.} 
\par
To avoid the free lunch issue, Sekitani and Zhao~\cite{Sekitani02032025} modified the minimization problem~\eqref{minRM} into the following maximization problem:
For any $(\bm{x},\bm{y})\in T_{\mbox{\tiny VRS-TO}}\cap \bbR^{m+s}_{++}$,
 \begin{align}
 \max\left\{\,  f(\bm{\theta},\bm{\phi}) \,\left|\,
  (\bm{\theta}\otimes\bm{x},\bm{\phi}\otimes\bm{y}) \in \partial^s(T_{\mbox{\tiny VRS-TO}}), \eqref{P5}
 \right\}.\right.\label{maxRM}
 \end{align}
 However, since they assumed positive inputs and positive outputs, 
 they did not overcome the boundary problem.
\par
Let 
\begin{align}
	T_{\mbox{\tiny CRS}}:=
	\left\{ 
		\left(\bm{x},\bm{y}\right) 
		\begin{array}{|l}
			\sum_{j=1}^n \lambda_j \bm{x}_j  \leq \bm{x}, \\
			\sum_{j=1}^n \lambda_j \bm{y}_j  \geq \bm{y}, \\
			\lambda_j \geq 0, \ j=1,\ldots,n \\
		\end{array}
	\right\}
	\cap  \left( \bar{\bbR}^{m}_{+} \times \bbR^{s}_{+} \right).   \label{CRSPPS}
\end{align}

\par
F{\"a}re and Lovell~\cite{fare1978measuring} proposed an input-oriented efficiency measure:
For any $(\bm{x},\bm{y})\in\Tcrs$,
\begin{align}
&\Efgl(\bm{x},\bm{y};T_{\mbox{\tiny CRS}}) := \nonumber\\
&\min\left\{\left. 
 \frac{\sum_{i=1}^m  \delta(x_i)\theta_i + \sum_{r=1}^s \delta(y_r)\cdot 1/\phi_r }{ \sum_{i=1}^m  \delta(x_i) + \sum_{r=1}^s \delta(y_r)  }
 \right| 
   (\bm{\theta}\otimes\bm{x},\bm{\phi}\otimes\bm{y}) \in T_{\mbox{\tiny CRS}}, \eqref{P5}\right\},\label{EFGL}
\end{align}
where 
\[
	\delta(z):=
	\begin{cases}
		0 & \text{if $z=0$};\\
		1 & \text{if $z>0$}.
	\end{cases}
\]
However, $\Efgl(\bm{x},\bm{y};T_{\mbox{\tiny CRS}})$ fails to satisfy both (I) and (WM) on the boundary of  $T_{\mbox{\tiny CRS}}$ with $y_r=0$.

Then, Levkoff et al.~\cite{levkoff2012boundary} modified the objective function of the F{\"a}re--Lovell efficiency measure 
$\Efgl(\bm{x},\bm{y};T_{\mbox{\tiny CRS}})$ as follows:
For any $(\bm{x},\bm{y})\in T_{\mbox{\tiny CRS}}$,
\begin{align}\label{barEFGL}
\barEfgl(\bm{x},\bm{y};T_{\mbox{\tiny CRS}})  &:= \inf\left\{\left. \bar{f}(\bm{\theta},\bm{\phi}) 
 \right| 
  (\bm{\theta}\otimes\bm{x},\bm{\phi}\otimes\bm{y}) \in T_{\mbox{\tiny CRS}}, \eqref{P5}
 \right\},
\end{align} 
where
\begin{align}
\bar{f}(\bm{\theta},\bm{\phi}) := 
 \frac{\sum_{i=1}^m  \delta(x_i)\theta_i + \sum_{r=1}^s \psi_r(\bm{x},\bm{y};T_{\mbox{\tiny CRS}})\cdot 1/\phi_r }{ \sum_{i=1}^m  \delta(x_i) + \sum_{r=1}^s \psi_r(\bm{x},\bm{y};T_{\mbox{\tiny CRS}})  },
\end{align}
and 
\[
\psi_r(\bm{x},\bm{y};T_{\mbox{\tiny CRS}}) :=
\left\{ \begin{array}{ll} 0 & \mbox{if } y_r=0 \mbox{ and } (\bm{x},\bm{y}+\epsilon\bm{e}_r) \not\in T_{\mbox{\tiny CRS}}   \mbox{ for any } \epsilon>0;
		\\ 1 & \mbox{otherwise}.
\end{array}
\right.
\]
The modified version $\barEfgl(\cdot,\cdot;T_{\mbox{\tiny CRS}})$ satisfies (I) and (WM).

From the definitions of~\eqref{EFGL} and~\eqref{barEFGL},
for each $(\bm{x},\bm{y})\in  T_{\mbox{\tiny CRS}} \cap \left( \bar{\bbR}^m_{+} \times \bbR^s_{++}\right)$, we have
\begin{align}
\barEfgl(\bm{x},\bm{y};T_{\mbox{\tiny CRS}}) =
\Efgl(\bm{x},\bm{y};T_{\mbox{\tiny CRS}}),
\end{align}
and for each $(\bm{x},\bm{y})\in  T_{\mbox{\tiny CRS}}\cap \left( \bbR^{m}_{++}\times \bbR^s_{++} \right)$,
\begin{align}
H(\bm{x},\bm{y}; T_{\mbox{\tiny CRS}})=\barEfgl(\bm{x},\bm{y};T_{\mbox{\tiny CRS}})= \Efgl(\bm{x},\bm{y};T_{\mbox{\tiny CRS}}).
\end{align}
The modified version $\barEfgl(\cdot,\cdot;T_{\mbox{\tiny VRS-TO}})$ also satisfies 
the properties (I) and (WM)  if $(\bm{0},\bm{y})\not\in T_{\mbox{\tiny VRS-TO}}$ for any $\bm{y}\in \bbR^s_{++}$. 

Although the modified efficiency measure by Levkoff et al.~\cite{levkoff2012boundary} overcomes the boundary problem,
their model fails to provide a target.
For an inefficient DMU, there may exist no feasible solution to~\eqref{barEFGL} 
that attains the minimum, or efficiency score, and hence
no target representing an improvement from the current state for the inefficient DMU can be identified.

In this study, to enhance $\barEfgl(\cdot,\cdot;T_{\mbox{\tiny VRS-TO}})$ and $\Efgl(\cdot,\cdot;T_{\mbox{\tiny VRS-TO}})$,
	we adopt the maximization model~\eqref{maxRM} by~\cite{Sekitani02032025} and overcome the boundary problem.
We show the following properties: 
\begin{itemize}
	\item The maximization problem~\eqref{maxRM} has an optimal solution for any boundary point  $(\bm{x},\bm{y})\in T_{\mbox{\tiny VRS-TO}}$;
\item An efficiency measure  defined by the optimal value of~\eqref{maxRM} satisfies (I) and (SM);
\item The optimal value of~\eqref{maxRM} can be provided by solving $(m+s)$ linear programming problems.
\end{itemize}
 
\section{Extended Max Russell Graph Measure}\label{sec:effMeasure}
In this section, we demonstrate that the efficiency measure
induced from~\eqref{maxRM} has desired properties (I) and (SM)
under $\partial^s(P)=\partial^w(P)\mathrel{\red{\neq}}\red{\emptyset}$.
We also show that the computation of the efficiency score reduces to
solving a series of LPs.
\par
\leavevmode\red{Note that the assumption $\partial^s(P)=\partial^w(P)$ is satisfied in the following settings.
The PPS with $m=s=1$ and constant returns to scale satisfies this assumption, and
such PPSs are often presented in DEA textbooks~\cite{cooper2007data,hackman2008production}.
Moreover, an extension of a conventional PPS using the extended efficient facet approach%
~\cite{bessent1988efficiency,olesen2003identification,olesen2015facet}
also satisfies this assumption, and efficiency analyses using such extended PPSs
have been reported in~\cite{anouze2025redefining,aparicio2016survey,aparicio2020well,huang2025data,zhu2022determining}.}

We extend the domain $T_{\mbox{\tiny VRS-TO}}\cap \bbR^{m+s}_{++}$ of 
the corresponding efficiency measure for~\eqref{maxRM} into
$T_{\mbox{\tiny VRS-TO}}\cap \bbR^{m+s}_{+}$,
which is indeed identical  to $T_{\mbox{\tiny VRS-TO}}$.
For any $(\bm{x},\bm{y})\in T_{\mbox{\tiny VRS-TO}}$, we define 
\begin{align}
 \hat{F}(\bm{x},\bm{y};T_{\mbox{\tiny VRS-TO}}) := \mbox{  the optimal value of~\eqref{maxRM} },\label{emaxRM}
\end{align}
and we refer to $\hat{F}$ as an \emph{extended} max Russell graph measure
(max RGM).

The following lemma establishes the well-definedness of the proposed efficiency
measure $\hat F$ by showing that there exists an optimal solution to~\eqref{maxRM}
over $\Tvrs$.
\begin{lemma}\label{lemma1}
  For a {\rm PPS} $\Tvrs\in{\cal T}$,
	assume that $\partial^s(P)=\partial^w(P)\mathrel{\red{\neq}}\red{\emptyset}$.
  Then, for any $(\bm{x},\bm{y})\in  T_{\mbox{\tiny VRS-TO}}$,
  the extended max RGM $\hat{F}(\bm{x},\bm{y};T_{\mbox{\tiny VRS-TO}})$
  is well-defined, i.e., the optimal value of~\eqref{maxRM} exists.
  Moreover, the following strict inequality holds:
  \begin{align}\label{eq26}
    \hat{F}(\bm{x},\bm{y};\Tvrs)>1-\frac{1}{m+s}\red{.}
  \end{align}
\end{lemma}
\begin{proof}
Let
  \begin{align}\label{eq:nonzero.indices}
    \begin{gathered}
    I^+(\bm{x}):=\left\{\,i\mid x_i>0,\ i=1,\dots,m \right\} \text{ and } 
    I^+(\bm{y}):=\left\{\,r\mid y_r>0,\ r=1,\dots,s \right\}.
    \end{gathered}
  \end{align}
  By the assumption $\partial^s(P)=\partial^w(P)$, it follows from Lemma~\ref{lem:positive.vandu} that  
  $(\bm{v}^l,\bm{u}^l)\in\bbR^{m+s}_{++}$ for all $l=1,\ldots,L$.
  Let $(\bm{x},\bm{y})\in \Tvrs$ and 
  \[
  \phi^{\natural}_r := \min_{l=1,\ldots,L} \frac{\bm{v}^l\bm{x}-\sum_{q\not=r} u^l_q y_q - \sigma^l}{u^l_ry_r}\quad 
	\text{for all $r\in I^+(\bm{y})$}.
  \]
  Since $\bm{v}^l\bm{x}-\sum_{q\neq r}u^l_qy_q -\sigma^l\ge u^l_ry_r$ 
  for all $l=1,\dots,L$,  we have $\phi^\natural_r \ge 1$.
  By the definition of $\phi^\natural_r$ for all $r\in I^+(\bm{y})$,
  \begin{align*}
    \phi^\natural_r\le
    \frac{\bm{v}^l\bm{x}-\sum_{q\not=r} u^l_q y_q - \sigma^l}{u^l_ry_r}
    \quad \mbox{ for all }  l=1,\dots,L,
  \end{align*}
	and then a simple calculation leads to
  \begin{align*}
    & \bm{v}^l\bm{x}-\sum_{q\not=r} u^l_q y_q - \sigma^l\ge
    \phi^\natural_r u^l_ry_r \quad   \mbox{ for all }  l=1,\dots,L, \\
    \iff & \bm{v}^l(\bm{1}_m\otimes\bm{x})-
    \bm{u}^l\left(((\phi^\natural_r-1)\bm{e}_r+\bm{1}_s)\otimes\bm{y}\right)
    -\sigma^l \ge 0\quad \mbox{ for all }  l=1,\dots,L.
  \end{align*}
	Let $\bm{e}_p$ be a unit vector whose $p$th component is $1$. 
  By the definition of $\phi^\natural_r$, there exists $\bar{l}\in\{1,\dots,L\}$
  such that
  \begin{align}\label{eq:frontier}
  \bm{v}^{\bar{l}}(\bm{1}_m\otimes\bm{x})-
	\bm{u}^{\bar{l}}\left(((\phi^{\natural}_r-1)\bm{e}_r+\bm{1}_s)
  \otimes \bm{y}\right)-\sigma^{\bar{l}} = 0.
  \end{align}
  It follows from~\eqref{eq:frontier} and Lemma~\ref{lem:hyperplane} that
  $(\bm{1}_m\otimes\bm{x},((\phi^{\natural}_r-1)\bm{e}_r+\bm{1}_s)\otimes
  \bm{y}) \in \partial^s(T_{\mbox{\tiny VRS-TO}})$, and
	\begin{align}
		1-\frac{1}{m+s} & <
		\frac{m+s-1+1/\phi_r^{\natural}}{m+s} = \frac{
		\sum_{i=1}^m 1 + \sum_{q\not=r} 1 + 1/\phi_r^{\natural}}
		{m+s} \nonumber\\
		&= 
		f(\bm{1}_m,(\phi^{\natural}_r-1)\bm{e}_r+\bm{1}_s)\nonumber\\
		& \leq
				 \sup\left\{\,  f(\bm{\theta},\bm{\phi}) \,\left|\,
					 (\bm{\theta}\otimes\bm{x},\bm{\phi}\otimes\bm{y}) \in \partial^s(T_{\mbox{\tiny VRS-TO}}), \eqref{P5}
			 \right\}.\right. \label{ieq:supRM} 
	\end{align}
	Suppose that
  \begin{align*}
  \bm{\theta}\leq \bm{1}_m,\   \bm{\phi}\geq \bm{1}_s
  \mbox{  and } (\bm{\theta}\otimes\bm{x},\bm{\phi}\otimes \bm{y})
	\in \partial^s(\Tvrs).
  \end{align*}
	Then, for each $l=1,\dots,L$, we have
  $\bm{v}^l(\bm{\theta}\otimes\bm{x})-\bm{u}^l(\bm{\phi}\otimes\bm{y})
  -\sigma^l\ge 0$, and hence it follows from $\bm{\theta}\le\bm{1}_m$
  and $\bm{\phi}\ge\bm{1}_s$ that
  \begin{align*}
		0 \leq  \bm{v}^l(\bm{\theta}\otimes\bm{x})-\bm{u}^l(\bm{\phi}\otimes\bm{y})
    -\sigma^l
    \le
    \bm{v}^l\bm{x}-\sum_{q\neq r}u^l_qy_q-u^l_r\phi_ry_r-\sigma^l.
  \end{align*}
  Thus, for all $r\in I^+(\bm{y})$,
  \[
    \begin{aligned}
    &\phi_r\le\frac{\bm{v}^l\bm{x}-\sum_{q\neq r}u^l_qy_q-\sigma^l}{u^l_ry_r}
    \quad\text{for each $l=1,\dots,L$} \\
    \iff&
    \phi_r\le\min_{l=1,\dots,L}
    \frac{\bm{v}^l\bm{x}-\sum_{q\neq r}u^l_qy_q-\sigma^l}{u^l_ry_r}
    =\phi^\natural_r,
    \end{aligned}
  \]
  which implies that
  $1\leq \phi_r \leq \phi_r^{\natural}$ for all $r \in I^+(\bm{y})$.
  Therefore,  
  the problem~\eqref{ieq:supRM}  is equivalent to 
   \begin{align}
   \sup\left\{\,  f(\bm{\theta},\bm{\phi}) \,\left|\,
  \begin{array}{l} 
    (\bm{\theta}\otimes\bm{x},\bm{\phi}\otimes\bm{y}) \in \partial^s(T_{\mbox{\tiny VRS-TO}}),  \\
    \theta_i=1 \ (i\notin I^+(\bm{x})), \   0\leq \theta_i \leq 1    (i\in I^+(\bm{x})), \\
     \phi_r=1 \ (r\notin I^+(\bm{y})),\   1\leq \phi_r\leq  \phi^{\natural}_r \ (r\in I^+(\bm{y}))
  \end{array} 
   \right\}\right.. \label{OPTprob}
   \end{align}
	 This is because $\theta_ix_i=0$ holds for all $i \not\in I^+(\bm{x})$ and any $\theta_i\in [0,1]$,
	 and similarly, $\phi_r y_r=0$ holds for all $r\not\in I^+(\bm{y})$ and any $\phi_r\geq 1$.
  The feasible region of~\eqref{OPTprob} is nonempty and compact, and
  the objective function $f$ is continuous over the feasible region.
  Hence, the problem~\eqref{OPTprob} has an optimal solution;
  thus the maximum value of~\eqref{maxRM} exists.

  The latter assertion is readily shown due to the inequalities~\eqref{ieq:supRM}
  and the well-definedness of $\hat{F}(\bm{x},\bm{y};\Tvrs)$.
\end{proof}


As the following theorem indicates, the extended max RGM $\hat{F}$ avoids an
unrealistic target even if $\Tvrs$ allows free lunch, that is, $(\bm{0},\bm{y})\in\Tvrs$ for some $\bm{y}\in \bar{\bbR}^s_{+}$;
the condition on $\Tvrs$ that leads to the free lunch issue will be discussed in Section~\ref{sec:Comp}.

\begin{theorem}\label{Theorem2}
	Assume that $\partial^s(P)=\partial^w(P)\mathrel{\red{\neq}}\red{\emptyset}$.
  For $(\bm{x},\bm{y})\in T_{\mbox{\tiny VRS-TO}}$, let 
	$(\bm{\theta}^*,\bm{\phi}^*)$ be an optimal solution to~\eqref{maxRM}.
  Then, 
  \begin{align}
		\bm{\theta}^*\in \bbR^m_{++} \mbox{ and } \theta^*_ix_i>0 \mbox{ for each } i \in I^+(\bm{x}).\label{eq33}
  \end{align}
\end{theorem} 
\begin{proof}
  For any $(\bm{x},\bm{y})\in \Tvrs$, it follows from~\eqref{eq26} of Lemma~\ref{lemma1} that
  \begin{align*}
    &\hat{F}(\bm{x},\bm{y};\Tvrs)=
    \frac{1}{m+s}\left(
      \sum_{i=1}^m\theta^*_i+\sum_{r=1}^s\frac{1}{\phi^*_r}
    \right)>1-\frac{1}{m+s}=\frac{m+s-1}{m+s}\\
    &\iff
    \sum_{i=1}^m\theta^*_i+\sum_{r=1}^s\frac{1}{\phi^*_r}
    > m+s-1 \\
    &
    \begin{aligned}
    \iff\theta^*_i&>\left(m-\sum_{p\neq i}\theta^*_p\right)
    +\left(s-\sum_{r=1}^s \frac{1}{\phi^*_r}\right) - 1\\
    &\ge (m-(m-1))+(s-s) - 1 = 0 \mbox{\ for all } i=1,\ldots,m,
    \end{aligned}
  \end{align*}
	which means that $\bm{\theta}^*\in \bbR^{m}_{++}$.
  Therefore, $\theta^*_ix_i>0$ for each $i \in I^+(\bm{x}).$
\end{proof} 


The extended max RGM $\hat F$ has desirable properties (I) and (SM) as follows.
\begin{theorem}\label{Theorem1}
	Under $\partial^s(P)=\partial^w(P)\mathrel{\red{\neq}}\red{\emptyset}$,
  the  extended max RGM
  $\hat{F}$ defined by~\eqref{maxRM} satisfies \emph{(I)} and \emph{(SM)}.
\end{theorem}
\begin{proof}
\textbf{Proof of (I):}
For any given $(\bm{x},\bm{y})\in\Tvrs$,
		let $(\bm{\theta}^*,\bm{\phi}^*)$ be an optimal solution to~\eqref{maxRM}.
  Then, it follows from the definition~\eqref{emaxRM} of $\hat{F}$, $\bm{\theta}^*\leq \bm{1}_m$,  $(1/\phi_1^*,\ldots,1/\phi^*_s)\leq \bm{1}_s$, and 
  $(\bm{\theta}^*\otimes\bm{x},\bm{\phi}^*\otimes\bm{y}) \in \partial^s(T_{\mbox{\tiny VRS-TO}})$ that 
  \begin{align*}
    \hat{F}(\bm{x},\bm{y};\Tvrs)=1
    &\iff f(\bm{\theta}^*,\bm{\phi}^*)=\frac1{m+s}\left(\sum_{i=1}^m \theta^*_i
    +\sum_{r=1}^s\frac1{\phi^*_r}\right) = 1\\
    &\iff (\bm{\theta}^*,\bm{\phi}^*)=(\bm{1}_m,\bm{1}_s)\\
    &\iff (\bm{1}_m\otimes \bm{x},\bm{1}_s\otimes \bm{y})=(\bm{x},\bm{y})\in \partial^s(\Tvrs).
  \end{align*}

\textbf{Proof of (SM):}
Let $(\bm{\theta}^*,\bm{\phi}^*)$ be an optimal solution to~\eqref{maxRM}
  with $(\bm{x},\bm{y})\in\Tvrs$.
  Then, 
  \begin{align*}
  \bm{v}^l(\bm{\theta}^*\otimes \bm{x})-
  \bm{u}^l(\bm{\phi}^*\otimes\bm{y})-\sigma^l \geq 0,\ l=1,\ldots,L,
  \end{align*}
  and by $(\bm{\theta}^*\otimes\bm{x},\bm{\phi}^*\otimes\bm{y})\in\partial^s(\Tvrs)$ 
  and Lemma~\ref{lem:hyperplane}-\ref{hyper.a}, there exists $\bar{l}\in\{1,\dots,L\}$
  such that
  \begin{align*}
  \bm{v}^{\bar{l}}(\bm{\theta}^*\otimes \bm{x})-
  \bm{u}^{\bar{l}}(\bm{\phi}^*\otimes\bm{y})-\sigma^{\bar{l}} = 0. 
  \end{align*} 
  Consider $(\bm{x}',\bm{y}')\in T_{\mbox{\tiny VRS-TO}}$
  with $(\bm{x}',-\bm{y}')\leq (\bm{x},-\bm{y})$ and
  $(\bm{x}',-\bm{y}')\not= (\bm{x},-\bm{y})$.
	{\red{Since $(\bm{x},\bm{y})\notin \partial^s(\Tvrs)$, we have $\hat{F}(\bm{x},\bm{y};T_{\mbox{\tiny VRS-TO}})<1$.}}


  If  $(\bm{x}',\bm{y}')\in \partial^s(T_{\mbox{\tiny VRS-TO}})$, then  
  we have
  \begin{align}\label{ieq:ind}
  \hat{F}(\bm{x}',\bm{y}';T_{\mbox{\tiny VRS-TO}})=1 > \hat{F}(\bm{x},\bm{y};T_{\mbox{\tiny VRS-TO}}).
  \end{align}

  Otherwise, 
  Lemma~\ref{lem:hyperplane} guarantees that
	$0< \min_{l=1,\ldots,L} \mathopen{\red{(}}\bm{v}^l(\bm{1}_m\otimes\bm{x}')-\bm{u}^l(\bm{1}_s\otimes \bm{y}')-\sigma^l\mathclose{\red{)}}$,
	 and for some $\bar{l}\in \{1,\ldots,L\}$, it follows from~\eqref{eq33} and $\bm{\phi}^*\geq \bm{1}_s$ that  
  \[
  0= \bm{v}^{\bar{l}}(\bm{\theta}^*\otimes\bm{x})-\bm{u}^{\bar{l}}(\bm{\phi}^*\otimes\bm{y})-\sigma^{\bar{l}}>
	\bm{v}^{\bar{l}}(\bm{\theta}^*\otimes\bm{x}')-\bm{u}^{\bar{l}}(\bm{\phi}^*\otimes \bm{y}')-\sigma^{\bar{l}}. 
  \]
	This means that 
  \begin{align}
      &\bm{v}^{\bar{l}}\left(\bm{\theta}^*\otimes \bm{x}'\right)-\bm{u}^{\bar{l}} \left(\bm{\phi}^*\otimes \bm{y}'\right)< \sigma^{\bar{l}} < \bm{v}^{\bar{l}}\left(\bm{1}_m\otimes \bm{x}'\right) -\bm{u}^{\bar{l}}\left(\bm{1}_s\otimes \bm{y}'\right),   \label{eqFRAC1}\\
\intertext{and }
&\sigma^l  <  \bm{v}^l(\bm{1}_m\otimes\bm{x}')-\bm{u}^l(\bm{1}_s\otimes \bm{y}')\   \mbox{ for all } {l} =1,\ldots,L.  \label{eqFRAC2}
  \end{align}
  Let 
  \begin{align}
  &\alpha :=\notag\\
	&\max\left\{
  \frac{\sigma^{{l}} -\bm{v}^{{l}}(\bm{\theta}^*\otimes\bm{x}')
  +\bm{u}^{{l}}(\bm{\phi}^*\otimes \bm{y}')
  }{
  \bm{v}^{{l}}\left((\bm{1}_m-\bm{\theta}^*)\otimes \bm{x}'\right)
  -\bm{u}^{{l}}\left((\bm{1}_s-\bm{\phi}^*)\otimes \bm{y}'\right)
  }\ \middle|\ 
  \begin{array}{l}
		l=1,\ldots,L \mbox{ and }\\
  \sigma^{{l}} > \bm{v}^{{l}}(\bm{\theta}^*\otimes\bm{x}')
  -\bm{u}^{{l}}(\bm{\phi}^*\otimes \bm{y}')
  \end{array}
  \right\}.\notag 
  \end{align}
  Then, it follows from~\eqref{eqFRAC1} and~\eqref{eqFRAC2}  that  $0<\alpha<1$. 
   This means that  $((\alpha \bm{1}_m+(1-\alpha)\bm{\theta}^*)\otimes\bm{x}',
  (\alpha \bm{1}_s+(1-\alpha)\bm{\phi}^*)\otimes\bm{y}')
  \in T_{\mbox{\tiny VRS-TO}}$, and
  for some $\hat{l} \in \{1,\ldots,L\}$, we have
  \[
  \bm{v}^{\hat{l}} ((\alpha \bm{1}_m+(1-\alpha)\bm{\theta}^*)\otimes\bm{x}')- \bm{u}^{\hat{l}} ((\alpha \bm{1}_s+(1-\alpha)\bm{\phi}^*)\otimes\bm{y}') - \sigma^{\hat{l}}=0. 
  \]
  Therefore, we have $((\alpha \bm{1}_m+(1-\alpha)\bm{\theta}^*)\otimes\bm{x}',
  (\alpha \bm{1}_s+(1-\alpha)\bm{\phi}^*)\otimes\bm{y}')
  \in \partial^s(T_{\mbox{\tiny VRS-TO}})$.
	Since the objective function $f(\bm{\theta},\bm{\phi})$ is increasing
	with respect to  $\bm{\theta}$ and decreasing with respect to $\bm{\phi}$,
  it follows from $(-\alpha (\bm{1}_m-\bm{\theta}^*)-\bm{\theta}^*,
  \alpha(\bm{1}_s-\bm{\phi}^*)+\bm{\phi}^*) \leq (-\bm{\theta}^*,\bm{\phi}^*)$
  and $(-\alpha (\bm{1}_m-\bm{\theta}^*)-\bm{\theta}^*,
  \alpha(\bm{1}_s-\bm{\phi}^*)+\bm{\phi}^*) \not= (-\bm{\theta}^*,\bm{\phi}^*)$
  that
  \begin{align*}
		\hat{F}(\bm{x}',\bm{y}';\Tvrs) & \geq f(\alpha \bm{1}_m+(1-\alpha)\bm{\theta}^*,\alpha \bm{1}_s+(1-\alpha)\bm{\phi}^*)\\
																	 & >f(\bm{\theta}^*,\bm{\phi}^*)=\hat{F}(\bm{x},\bm{y};\Tvrs). 
  \end{align*}
\end{proof} 

The following theorem suggests that it is enough to solve ($m+s$) LPs
to compute the efficiency score $\hat{F}(\bm{x},\bm{y};\Tvrs)$.

\begin{theorem}\label{thm:closest}
	Assume $\partial^s(P)=\partial^w(P)\mathrel{\red{\neq}}\red{\emptyset}$.
  For any $(\bm{x},\bm{y})\in\Tvrs$, let $I^+(\bm{x})$ and $I^+(\bm{y})$
  be the sets of indices defined by~\eqref{eq:nonzero.indices},
  and let
	
	\begin{align}\label{theta.star}
    \theta^*:=
    \begin{cases}
\, \red{0}  & \text{if $I^+(\bm{x})=\emptyset$}; \\
      \displaystyle
      \max_{i \in I^+(\bm{x})}\min
      \left\{\theta_i\ \middle|\
      \begin{aligned}
         & (\bm{\theta}\otimes\bm{x},\bm{y})\in \Tvrs, \\
         & \theta_q=1\quad \forall q \not=i
      \end{aligned}
      \right\}
                       & \text{otherwise},
    \end{cases}
  \end{align}
  and
	\begin{align}\label{phi.star}
    \phi^*:=\displaystyle
    \min_{r \in I^+(\bm{y})}
    \max\{\phi_r\mid
    (\bm{x},\bm{\phi}\otimes\bm{y})\in T_{\mbox{\tiny VRS-TO}},\ \phi_q=1
    \ \forall q \not=r\}.
  \end{align}
  Then,
  \begin{align}
    \hat{F}(\bm{x},\bm{y};T_{\mbox{\tiny VRS-TO}})=
    \max\left\{
    \frac{m+s-1+\theta^*}{m+s}, \
    \frac{m+s-1+{1}/{\phi^*}}{m+s} 
    \right\}.                     \label{eqClosest}
  \end{align}
\end{theorem}
\begin{proof}


	\leavevmode\red{For $(\bm{x},\bm{y})\in \Tvrs$,
after relabeling the coordinates, we may assume without loss of generality that
$I^+(\bm x) = \{1,\dots,\bar m\}$ and $I^+(\bm y) = \{1,\dots,\bar s\}$,
where $\bar m$ and $\bar s$ are integers satisfying $0\le \bar m\le m$ and $1\le \bar s\le s$, respectively.
Write $\bm x=(\bar{\bm x},\bm 0)$ and $\bm y=(\bar{\bm y},\bm 0)$,
where $\bar{\bm x}\in\bbR^{\bar m}_{++}$ and $\bar{\bm y}\in \bbR^{\bar s}_{++}$, respectively.
Throughout this proof, the maximum over an empty index set is understood to be $-\infty$.

For every $l=1,\ldots,L$ and any $(\bm{v}^l,\bm{u}^l,\sigma^l)\in \bbR^{m+s}_{++}\times \bbR$  satisfying~\eqref{P}, let 
$\bar{\bm{v}}^l:=(v_1^l,\dots,v_{\bar{m}}^l)$ and
$\bar{\bm{u}}^l:=(u_1^l,\dots,u_{\bar{s}}^l)$ for all $l=1,\ldots,L$. Then, 
we have
\begin{align}
\left\{\, (\bar{\bm{x}},\bar{\bm{y}})\, \left|\, (\bar{\bm{x}},\bm{0};\bar{\bm{y}},\bm{0})\in \partial^s(P) \right\}\right. = 
 \bigcup_{\bar{l}=1}^L 
	\left\{ (\bar{\bm{x}},\bar{\bm{y}}) \left|
			\begin{array}{l}
				\bar{\bm{v}}^l \bar{\bm{x}} - \bar{\bm u}^l\bar{\bm{y}} \geq \sigma^l \quad \forall l=1,\ldots,L, \\
			\bar{\bm{v}}^{\bar{l}}\bar{\bm{x}} - \bar{\bm{u}}^{\bar{l}}\bar{\bm{y}}= \sigma^{\bar{l}} 
		\end{array}
	\right\}.\right.\label{eqOmega}
\end{align}
For every $i =1,\ldots,\bar{m}$, let
\begin{align*}
  \bar{\theta}^i_i
  := \max\left\{
    \frac{-\sum_{p\not=i} \bar{v}^l_p \bar{x}_p + \bar{\bm{u}}^l\bar{\bm{y}} + \sigma^l}
         {\bar{v}^l_i\bar{x}_i}
    \ \middle|\ l=1,\ldots,L
  \right\}
	\text{ and }
	\bar{\bm{\theta}}^i := \bm{1}_{\bar{m}}+(\bar{\theta}^i_i-1)\bm{e}_i.
\end{align*}
Then,  we have
\begin{align}\label{eqTh43a}
	\bar{\theta}^i_i=\min\{ \bar{\theta}_i \mid (\bar{\bm{\theta}}\otimes\bar{\bm{x}},\bm{0}; \bar{\bm{y}},\bm{0})\in P,\  \bar{\bm{\theta}} \leq \bm{1}_{\bar{m}} \} \text{ and }
	(\bm{\bar{\theta}}^i\otimes \bar{\bm{x}},\bm{0};\bm{1}_{\bar s}\otimes\bar{\bm{y}},\bm{0})\in \partial^s(P). 
\end{align}
Moreover, we have 
	\begin{align}
				&	(\bm{\bar{\theta}}^i\otimes \bar{\bm{x}},\bm{0};\bm{1}_{\bar s}\otimes\bar{\bm{y}},\bm{0})\in  \partial^s(\Tvrs) & \text{ if } \bar{\theta}^i_i \geq 0;
				\label{eq:target}\\
				&	((\bm{1}_{\bar{m}}-\bm{e}_i)\otimes \bar{\bm{x}},\bm{0}; \bm{1}_{\bar s}\otimes\bar{\bm{y}},\bm{0}) \in \partial^w(\Tvrs) & \text{otherwise}.\notag
      \end{align}
Hence, for every $i=1,\dots,\bar m$, 
\begin{align}
\min
      \left\{\theta_i\ \middle|\
         (\bm{\theta}\otimes\bm{x},\bm{y})\in \Tvrs,\,  \theta_p=1\ \ \forall p \not=i
      \right\}
=\max\left\{0, \bar{\theta}^i_i  \right\}.\label{eqThetaI}
\end{align}

Similarly, for every $r =1,\ldots,\bar{s}$, let
\begin{align*}
	\bar{\phi}^r_r := \min\left\{
    \frac{\bar{\bm{v}}^l \bar{\bm{x}} - \sum_{q\not=r} \bar{u}^l_q \bar{y}_q - \sigma^l}
         {\bar{u}^l_r\bar{y}_r}
    \ \middle|\ l=1,\ldots,L
  \right\}
	\text{ and }
	\bar{\bm{\phi}}^r  := \bm{1}_{\bar{s}}+(\bar{\phi}^r_r-1)\bm{e}_r.
\end{align*}
Then, we have 
\begin{align}\label{eqTh43b}
	\begin{aligned}
		&\bar{\phi}^r_r=\max\{ \bar{\phi}_r \mid (\bar{\bm{x}},\bm{0}; \bar{\bm{\phi}}\otimes \bar{\bm{y}},\bm{0})\in P,\ \bm{1}_{\bar s}\leq \bar{\bm{\phi}} \}
		\text{ and }  
		(\bm{1}_{\bar m}\otimes \bar{\bm{x}},  \bm{0};\bm{\bar{\phi}}^r\otimes \bar{\bm{y}},\bm{0})\in \partial^s(\Tvrs), 
	\end{aligned}
\end{align}
and  hence we have
\begin{align}
    \max\{\phi_r\mid
    (\bm{x},\bm{\phi}\otimes\bm{y})\in T_{\mbox{\tiny VRS-TO}},\ \phi_q=1
    \ \forall q \not=r\}
= \bar{\phi}^r_r. \label{eqPhiR}
\end{align}

Let 
\begin{gather*}
  \bar{f}(\bar{\bm{\theta}},\bar{\bm{\phi}})
  := \sum_{i\in I^+(\bm{x})} \bar{\theta}_i + \sum_{r\in I^+(\bm{y})} \frac{1}{\bar{\phi}_r}
	=\sum_{i=1}^{\bar{m}} \bar{\theta}_i+  \sum_{r=1}^{\bar{s}} \frac{1}{\bar{\phi}_r}\quad \text{ and } 
  \\
  \bar{\Delta} := \left\{
    \sum_{i=1}^{\bar{m}} \mu_i (\bar{\bm{\theta}}^i,\bm{1}_{\bar{s}})
    + \sum_{r=1}^{\bar{s}} \mu_{r+\bar{m}}(\bm{1}_{\bar{m}},\bar{\bm{\phi}}^r)
    \ \middle|\ \sum_{p=1}^{\bar{m}+\bar{s}} \mu_p =1, \bm{\mu} \geq \bm{0}
  \right\}.
\end{gather*}
Then, we have
\begin{align}
(\bar{\bm{\theta}}^{\sharp}\otimes \bar{\bm{x}},\bm{0}; \bar{\bm{\phi}}^{\sharp}\otimes \bar{\bm{y}},\bm{0})\in P
\mbox{ for any } (\bar{\bm{\theta}}^{\sharp},\bar{\bm{\phi}}^{\sharp})\in \bar{\Delta}.\label{eqDelta0}
\end{align}
Since $\bar{\Delta}$ is the convex hull of $(\bar{\bm\theta}^i,\bm{1}_{\bar s})$,
$i=1,\dots,\bar m$, $(\bm{1}_{\bar m},\bar{\bm\phi}^r)$, $r=1,\dots,\bar s$,
and $\bar f$ is a quasi-convex function, we have 
\begin{align}
	\max\left\{
		\bar{f}(\bar{\bm{\theta}},\bar{\bm{\phi}})
		\ \middle|\ (\bar{\bm{\theta}},\bar{\bm{\phi}}) \in \bar{\Delta}
	\right\}
	= 
	\max\left\{
		\max_{i=1,\ldots,\bar{m}} \bar{f}(\bm{\bar{\theta}}^i,\bm{1}_{\bar{s}}),\
		\max_{r=1,\ldots,\bar{s}} \bar{f}(\bm{1}_{\bar{m}}, \bm{\bar{\phi}}^r)
	\right\}. \label{eqTh43c}
\end{align}


Let $\bar{\bm{\theta}}^0:=(\bar{\theta}^1_1,\cdots,\bar{\theta}^{\bar{m}}_{\bar{m}})^{\top}, \  \bar{\bm{\phi}}^0:=(\bar{\phi}^1_1,\cdots,\bar{\phi}^{\bar{s}}_{\bar{s}})^{\top},$ and 
\begin{align*}
&\Omega := \bigcup_{\bar{l}=1}^L 
	\left\{ (\bar{\bm{\theta}},\bar{\bm{\phi}}) \left|
			\begin{array}{l}
			\bar{\bm{v}}^l (\bar{\bm{\theta}}\otimes\bar{\bm{x}}) - \bar{\bm{u}}^l (\bar{\bm{\phi}}\otimes\bar{\bm{y}})
			\geq \sigma^l \quad \forall l=1,\ldots,L, \\
			\bar{\bm{v}}^{\bar{l}} (\bar{\bm{\theta}}\otimes\bar{\bm{x}}) - \bar{\bm{u}}^{\bar{l}} (\bar{\bm{\phi}}\otimes\bar{\bm{y}})
			= \sigma^{\bar{l}}, \\
			\bar{\theta}_i \leq 1 \ \ (i=1,\ldots,\bar{m}),\ 1\leq \bar{\phi}_r \ \ (r=1,\ldots,\bar{s})
		\end{array}
	\right\}.\right.
\end{align*}
Then,  we have 
\begin{align}
&(\bar{\bm\theta}^i,\bm 1_{\bar s})\in\Omega \mbox{ for all $i=1,\dots,\bar m$}, \ (\bm 1_{\bar m},\bar{\bm\phi}^r)\in\Omega \mbox{  for all $r=1,\dots,\bar s$,}\label{eq421}
\end{align}
and 
$\bm{1}_{{\bar m}} \geq \bar{\bm{\theta}}' \geq \bar{\bm{\theta}}^0$, $\bm{1}_{{\bar s}} \leq  \bar{\bm{\phi}}' \leq \bar{\bm{\phi}}^0$ for any $ (\bar{\bm{\theta}}^{\prime},\bar{\bm{\phi}}^\prime) \in \Omega$. 
When $(\bm{x},\bm{y})\in \partial^s(P)$, we have 
$\bar{\bm{\theta}}^0= \bm{1}_{\bar m}$ and $\bar{\bm{\phi}}^0=\bm{1}_{\bar s}$, and thus
\begin{align}
&\bar{\bm{\theta}}'=\bm{1}_{\bar m} \mbox{ and } \bar{\bm{\phi}}'=\bm{1}_{\bar s}
\mbox{ for any }  (\bar{\bm{\theta}}^\prime,\bar{\bm{\phi}}^\prime)\in \Omega.\label{eqbar01}
\end{align}
It follows from~\eqref{eqTh43c} and \eqref{eqbar01} that 
\begin{align}
	\max\left\{
		\bar{f}(\bar{\bm{\theta}},\bar{\bm{\phi}})
		\ \middle|\ (\bar{\bm{\theta}},\bar{\bm{\phi}}) \in \bar{\Delta}
	\right\}
	= \bar{f}(\bm{1}_{\bar m},\bm{1}_{\bar s})=
	\max\left\{
		\bar{f}(\bar{\bm{\theta}},\bar{\bm{\phi}})
		\ \middle|\ (\bar{\bm{\theta}},\bar{\bm{\phi}}) \in \Omega
	\right\}.\label{eqDeltaOmega0}
\end{align}
\par
When $(\bm{x},\bm{y}) \not\in \partial^s(P)$, $\partial^w(P)=\partial^s(P)$ implies that
$\bar{\bm{\theta}}^0< \bm{1}_{\bar m} \mbox{ and }   \bar{\bm{\phi}}^0>\bm{1}_{\bar s}.$
Choose any  $ (\bar{\bm{\theta}}',\bar{\bm{\phi}}') \in \Omega$. Then, we have  $ (\bar{\bm{\theta}}',\bar{\bm{\phi}}')\not= (\bm{1}_{\bar m},\bm{1}_{\bar s})$.
Let 
\begin{gather*}
	\lambda' := \frac{1}{ \sum_{i=1}^{\bar m}  \frac{\bar{\theta}'_i-1}{\bar{\theta}^i_i-1 } + \sum_{r=1}^{\bar s} \frac{\bar{\phi}'_r-1}{\bar{\phi}^r_r-1} },\\	
	\mu_i^{\prime}:=  \frac{\bar{\theta}'_i-1}{\bar{\theta}^i_i-1}\lambda',\ i=1,\ldots,\bar m, \mbox{ and }\ 
	\mu_{r+\bar m}^{\prime}:=      \frac{\bar{\phi}'_r-1}{\bar{\phi}^r_r-1}\lambda',
	\ r=1,\ldots,\bar s.
\end{gather*}	
  Then, we have  $\bm{\mu}^{\prime}\geq \bm{0}$ and $\sum_{i=1}^{\bar m}\mu^{\prime}_i + \sum_{r=1}^{\bar s}\mu^{\prime}_{r+\bar m}=1$.
   Hence, we have 
$   \sum_{i=1}^{\bar{m}} \mu'_i (\bar{\bm{\theta}}^i,\bm{1}_{\bar{s}})
    + \sum_{r=1}^{\bar{s}} \mu'_{r+\bar{m}}(\bm{1}_{\bar{m}},\bar{\bm{\phi}}^r)=
(\lambda'(\bar{\bm{\theta}}'-\bm{1}_{\bar m})+\bm{1}_{\bar m}, \lambda'(\bar{\bm{\phi}}'-\bm{1}_{\bar s})+\bm{1}_{\bar s})\in\bar{\Delta}$.
It follows from~\eqref{eqDelta0} and $(\bar{\bm{\theta}}^\prime\otimes \bar{\bm{x}},\bm{0};\bar{\bm{\phi}}^\prime\otimes\bar{\bm{y}},\bm{0})\in \partial^s(P)$ that $\lambda'\in (0,1]$ and 
$(-\lambda'(\bar{\bm{\theta}}'-\bm{1}_{\bar m})-\bm{1}_{\bar m}, \lambda'(\bar{\bm{\phi}}'-\bm{1}_{\bar s})+\bm{1}_{\bar s})= (-\lambda'\bar{\bm{\theta}}'-(1-\lambda')\bm{1}_{\bar m}, \lambda'\bar{\bm{\phi}}'+(1-\lambda')\bm{1}_{\bar s}) \leq
(-\bar{\bm{\theta}}', \bar{\bm{\phi}}')$.
By the definition of $\bar f$, we have
		$\bar{f}(\bar{\bm{\theta}}^a,\bar{\bm{\phi}}^a)\geq \bar{f}(\bar{\bm{\theta}}^b,\bar{\bm{\phi}}^b)$
	for $(-\bar{\bm{\theta}}^a,\bar{\bm{\phi}}^a)\leq (-\bar{\bm{\theta}}^b,\bar{\bm{\phi}}^b)$, which yields
\begin{align}\label{ieq:max}
	\bar f(\bar{\bm\theta}^\prime,\bar{\bm\phi}^\prime) \le \bar f(\lambda'(\bar{\bm{\theta}}'-\bm{1}_{\bar m})+\bm{1}_{\bar m}, \lambda'(\bar{\bm{\phi}}'-\bm{1}_{\bar s})+\bm{1}_{\bar s})
	\le 	\max \left\{\bar f(\bar{\bm\theta},\bar{\bm\phi}) \mid (\bar{\bm\theta},\bar{\bm\phi})\in\bar{\Delta} \right\}.
\end{align}
Then, it follows from~\eqref{eq421} and~\eqref{ieq:max} that 

\begin{align}
\max\left\{
		\bar{f}(\bar{\bm{\theta}},\bar{\bm{\phi}})
		\ \middle|\ (\bar{\bm{\theta}},\bar{\bm{\phi}}) \in \bar{\Delta}
	\right\}
	= \max\left\{
	\begin{aligned}
& \max_{i=1,\ldots,\bar{m}} \bar{f}(\bm{\bar{\theta}}^i,\bm{1}_{\bar{s}}),\\
&   \max_{r=1,\ldots,\bar{s}} \bar{f}(\bm{1}_{\bar{m}}, \bm{\bar{\phi}}^r)
	\end{aligned}
  \right\}=
	\max\left\{
		\bar{f}(\bar{\bm{\theta}},\bar{\bm{\phi}})
		\ \middle|\ (\bar{\bm{\theta}},\bar{\bm{\phi}}) \in \Omega 
	\right\}. \label{eqDeltaOmega1}
\end{align}
Therefore, it follows from \eqref{eqDeltaOmega0} and \eqref{eqDeltaOmega1} that 
\begin{align}
  &\max\left\{
    \sum_{i\in I^+(\bm{x})} \theta_i + \sum_{r\in I^+(\bm{y})} \frac{1}{\phi_r}
    \ \middle|\
    \begin{array}{l}
      (\bm{\theta}\otimes\bm{x},\bm{\phi}\otimes\bm{y}) \in \partial^s(T_{\mbox{\tiny VRS-TO}}), \\
      \theta_i=1 \ (i\notin I^{+}(\bm{x})),\ \phi_r=1 \ (r\notin I^{+}(\bm{y})),\\
      0 \leq \theta_i \leq 1 \ (i=1,\ldots,m),\ 1\leq \phi_r \ (r=1,\ldots,s)
    \end{array}
  \right\} \label{A0} \\
  &= \max\left\{
    \bar{f}(\bar{\bm{\theta}},\bar{\bm{\phi}})
    \ \middle|\
    \begin{array}{l}
      (\bar{\bm{\theta}}\otimes\bar{\bm{x}},\bm{0};\bar{\bm{\phi}}\otimes\bar{\bm{y}},\bm{0}) \in \partial^s(T_{\mbox{\tiny VRS-TO}}),\\
      \theta_i=1 \ (i=\bar{m}+1,\ldots,m),\ \phi_r=1 \ (r=\bar{s}+1,\ldots,s),\\
      0 \leq \bar{\theta}_i \leq 1 \ (i=1,\ldots,\bar{m}),\ 1\leq \bar{\phi}_r \ (r=1,\ldots,\bar{s})
    \end{array}
  \right\} \notag \\
  &= \max\left\{
    \bar{f}(\bar{\bm{\theta}},\bar{\bm{\phi}})
    \ \middle|\
    \begin{array}{l}
      (\bar{\bm{\theta}}\otimes\bar{\bm{x}},\bm{0};\bar{\bm{\phi}}\otimes\bar{\bm{y}},\bm{0}) \in \partial^s(T_{\mbox{\tiny VRS-TO}}), \\
      0 \leq \bar{\theta}_i \leq 1 \ (i=1,\ldots,\bar{m}),\ 1\leq \bar{\phi}_r \ (r=1,\ldots,\bar{s})
    \end{array} 
	\right\} \label{A1.rev}\\
 &= \max\left\{
    \bar{f}(\bar{\bm{\theta}},\bar{\bm{\phi}})
    \ \middle|\
    \begin{array}{l}
      \bar{\bm{v}}^l (\bar{\bm{\theta}}\otimes\bar{\bm{x}})-\bar{\bm{u}}^l (\bar{\bm{\phi}}\otimes\bar{\bm{y}}) \geq \sigma^l \quad \forall l=1,\ldots,L, \\
			\bar{\bm{v}}^{\bar{l}} (\bar{\bm{\theta}}{\otimes}\bar{\bm{x}})-\bar{\bm{u}}^{\bar{l}} (\bar{\bm{\phi}}\otimes\bar{\bm{y}}) = \sigma^{\bar{l}}\quad \exists \bar{l}\in \{1,\ldots,L\}, \\
      0 \leq \bar{\theta}_i \leq 1 \ (i=1,\ldots,\bar{m}),\ 1\leq \bar{\phi}_r \ (r=1,\ldots,\bar{s})
    \end{array}
	\right\}\ \text{($\because$ Lemma~\ref{lem:hyperplane})} \notag \\
  &\leq \max\left\{
    \bar{f}(\bar{\bm{\theta}},\bar{\bm{\phi}})
    \ \middle|\
  (\bar{\bm{\theta}},\bar{\bm{\phi}}) \in \Omega
  \right\} \notag \\
  &= \max\left\{
    \max_{i=1,\ldots,\bar{m}} \bar{f}(\bm{\bar{\theta}}^i,\bm{1}_{\bar{s}}),\
    \max_{r=1,\ldots,\bar{s}} \bar{f}(\bm{1}_{\bar{m}}, \bm{\bar{\phi}}^r)
  \right\}. \label{A1} 
\end{align}
If  $\bar{\theta}^i_i \le 0$, then  
\[
\bar{f}(\bar{\bm\theta}^i,\bm{1}_{\bar s}) = \bar m + \bar s - 1 + \bar{\theta}^i_i \le \bar m + \bar s - 1 < 
\max_{r=1,\ldots,\bar{s}}(\bar{m}+\bar{s}-1+1/\bar{\phi}^r_r) =\max_{r=1,\dots,\bar s} \bar{f}(\bm{1}_{\bar m},\bar{\bm{\phi}}^r)
\]
and the maximum in~\eqref{A1} is equal to   
\begin{align}
	\max\left\{
    \max_{i: \bar{\theta}^i_i>0} \bar{f}(\bm{\bar{\theta}}^i,\bm{1}_{\bar{s}}),\
    \max_{r=1,\ldots,\bar{s}} \bar{f}(\bm{1}_{\bar{m}}, \bm{\bar{\phi}}^r)
	\right\}. \label{A99} 
\end{align}
Since   $(\bar{\bm\theta}^i,\bm{1}_{\bar s})$ for any  $i$ satisfying  $\bar{\theta}^i_i > 0$ and 
$(\bm{1}_{\bar{m}}, \bm{\bar{\phi}}^r)$ for all $r=1,\ldots,\bar{s}$ are all feasible solutions to~\eqref{A1.rev}, the maximum in~\eqref{A99} does not exceed the maximum in~\eqref{A1.rev}. 
Together with~\eqref{A0}--\eqref{A1.rev}, this shows that the maxima in~\eqref{A0} and~\eqref{A99} are equal.

 It follows from the  definitions of  $\theta^*$ in~\eqref{theta.star} and $\phi^*$ in~\eqref{phi.star}   that  
\begin{align*}
& \theta^*  = \max\left\{ 0,  \max_{i=1,\dots,\bar m}
	\min\left\{ \theta \ \middle|
     \begin{array}{l}
	   ((\bm{1}_m+(\theta-1)\bm{e}_i)\otimes\bm{x}, \bm{y})\in P\\
                  0\leq \theta \le 1
    \end{array}
	\right\}\right\}
	= \max\left\{0, \max_{i: \bar{\theta}^i_{{i}}>0} \bar{\theta}^i_i\right\},
	\intertext{and}
&\phi^*  = 
	\min_{r=1,\ldots,\bar{s}} \max\left\{
    \phi \ \middle|\ (\bm{x}, (\bm{1}_s +(\phi-1)\bm{e}_r)\otimes \bm{y})\in P,\ 1\le \phi
  \right\}
	= \min_{r=1,\ldots,\bar{s}} \bar{\phi}^r_r \geq 1.
\end{align*}
Since 
\begin{align*}
&\begin{aligned}
	\max_{i: \bar{\theta}^i_i>0} \bar{f}(\bar{\bm\theta}^i,\bm{1}_{\bar s})
	= \max_{i: \bar{\theta}^i_i>0} (\bar{m}+\bar{s}-1+\bar{\theta}^i_i)
	= \bar{m}+\bar{s}-1+ \max_{i: \bar{\theta}^i_i>0} \bar{\theta}^i_i 
	\end{aligned} 
     \intertext{and}
       &\begin{aligned}
		\max_{r=1,\dots,\bar s} \bar{f}(\bm{1}_{\bar m},\bar{\bm{\phi}}^r)
		= \max_{r=1,\ldots,\bar{s}} (\bar{m}+\bar{s}-1+1/\bar{\phi}^r_r)
		= \bar{m}+\bar{s}-1+1/\phi^*\\
	\end{aligned}, 	
\end{align*}
we have 
\begin{align}
  & \hat{F}(\bm{x},\bm{y};\Tvrs) \notag \\
 =&\max\left\{
    \frac{\sum_{i=1}^m \theta_i + \sum_{r=1}^s {1}/{\phi_r}}{m+s}
   \ \middle|\
    (\bm{\theta}\otimes\bm{x},\bm{\phi}\otimes\bm{y}) \in \partial^s(T_{\mbox{\tiny VRS-TO}}),\ \eqref{P5}
  \right\} \notag \\
  =&\max\left\{
    \frac{\sum_{i=1}^m \theta_i + \sum_{r=1}^s {1}/{\phi_r}}{m+s}
    \ \middle|
    \begin{array}{l}
      (\bm{\theta}\otimes\bm{x},\bm{\phi}\otimes\bm{y}) \in \partial^s(T_{\mbox{\tiny VRS-TO}}),\ \eqref{P5}, \\
      \theta_i=1 \ (i\notin I^+(\bm{x})), \ \phi_r=1 \ (r\notin I^+(\bm{y}))
   \end{array}
  \right\} \notag \\
  =&\max\left\{
    \begin{array}{l}
			\frac{\sum_{i=1}^{\bar m} \theta_i + \sum_{r=1}^{\bar s} {1}/{\phi_r}}{m+s} \\
      + \frac{\sum_{i=\bar m + 1}^m \theta_i + \sum_{r=\bar s + 1}^s {1}/{\phi_r}}{m+s}
    \end{array}
    \middle|
    \begin{array}{l}
      (\bm{\theta}\otimes\bm{x},\bm{\phi}\otimes\bm{y}) \in \partial^s(T_{\mbox{\tiny VRS-TO}}),\ \eqref{P5}, \\
      \theta_i=1 \ (i=\bar m+1,\dots,m),\ \phi_r=1 \ (r=\bar s+1,\dots,s)
    \end{array}
  \right\} \notag \\
  =&\frac{m-\bar m+s-\bar s}{m+s} + {} \notag \\
  &\max\left\{
		\frac{\sum_{i=1}^{\bar m} \theta_i + \sum_{r=1}^{\bar s} {1}/{\phi_r}}{m+s}
    \ \middle|
    \begin{array}{l}
      (\bm{\theta}\otimes\bm{x},\bm{\phi}\otimes\bm{y}) \in \partial^s(T_{\mbox{\tiny VRS-TO}}),\ \eqref{P5}, \\
      \theta_i=1 \ (i=\bar m+1,\dots,m), \ \phi_r=1 \ (r=\bar s+1,\dots,s)
    \end{array}
	\right\} \notag \\
	=&\frac{m-\bar m+s-\bar s}{m+s} + 
	\max\left\{ \max_{i:{\bar{\theta}^i_i}>0} \frac{\bar{f}(\bar{\bm\theta}^i,\bm{1}_{\bar s})}{m+s},
	\max_{r=1,\dots,\bar s} \frac{\bar{f}(\bm{1}_{\bar m},\bar{\bm\phi}^r)}{m+s}\right\} \notag \\
	=&\frac{m-\bar m+s-\bar s}{m+s} + 
	\max\left\{  \frac{\bar m + \bar s -1+ \max_{i:{\bar{\theta}^i_i}>0}  \bar{\theta}^i_i}{ m +  s},\   \frac{\bar m + \bar s -1+ 1/\phi^*}{m+s}\right\} \notag \\
	=& \frac{m + s - 1 + \max\{\theta^*, 1/\phi^*\}}{m+s}.\notag  
\end{align}
The last equality remains valid when $\{i\mid \bar{\theta}^i_i > 0\}=\emptyset$,
	since $\theta^*=\max\{0,\max_{i:\bar{\theta}^i_i>0}\bar{\theta}^i_i \} = \max\{0,-\infty\} = 0$ in this case and 
$1/\phi^* > 0$.
The proof is completed.}
\end{proof}


\leavevmode\red{Theorem~\ref{thm:closest} implies  that     $ \hat{F}(\bm{x},\bm{y};T_{\mbox{\tiny VRS-TO}})=   \frac{m+s-1+\theta^*}{m+s}$ or    $ \hat{F}(\bm{x},\bm{y};T_{\mbox{\tiny VRS-TO}})=   \frac{m+s-1+1/\phi^*}{m+s}$.
In the case 
$ \hat{F}(\bm{x},\bm{y};T_{\mbox{\tiny VRS-TO}})=   \frac{m+s-1+\theta^*}{m+s}$, it follows from Lemma~\ref{lemma1} that $\theta^*>0$,
and hence by~\eqref{eq:target}, we have
	 \begin{align}\label{eq:target2}
		 &((\bm{1}_{m}+(\theta^*-1)\bm{e}_i)\otimes \bm{x},\bm{y}) \notag \\
		 =& (\bm{x}-(1-\theta^*)\bm{e}_i\otimes\bm{x},\bm y) \in \partial^s(\Tvrs) \text{ for some $i\in I^+(\bm{x})$}.
	 \end{align}
 Then, $\partial^s(\Tvrs)\subseteq \partial^s(P)=\partial^w(P)$ and Briec~\cite[Corollary~3(2)]{Briec1999holder} ensure that
	 \begin{align}
	\theta^* &= \max_{i\in I^+(\bm x)}  \min\left\{ \theta_i \middle| \,
			(\bm{\theta}\otimes \bm{x},\bm{y})\in P,\ \theta_q=1\ \forall q \not=i 
		\right\} \notag \\
		&=1-\min_{i\in I^+(\bm x)}  \max\left\{ \delta^-_i \middle| \,
								 (\bm{x}-\delta^-_ix_i\bm{e}_i,\bm{y})\in P, \, \delta^-_i  \geq 0 
							 \right\} \notag \\
             &= 1- \min\left\{ \, \sum_{i\in I^+(\bm x)}  \delta^-_i \middle|\,
                           (\bm{x}-\bm{\delta}^-\otimes \bm{x},\bm{y})\in \partial^w(P),\, \bm{\delta}^- \geq \bm{0} \right\} \notag \\
            &=  1- \min\left\{ \, \sum_{i\in I^+(\bm x)}  \delta^-_i \middle|\,
                      (\bm{x}-\bm{\delta}^-\otimes \bm{x},\bm{y})\in \partial^s(P),\, \bm{\delta}^- \geq \bm{0} 
                   \right\}  \notag \\ 
            &\geq 1-  \min\left\{ \, \sum_{i=1}^m \delta^-_i \middle|\,
                  (\bm{x}-\bm{\delta}^-\otimes \bm{x},\bm{y})\in \partial^s(P),\, \bm{\delta}^- \geq \bm{0} 
          \right\}  \notag     \\
          &\geq  1-  \min\left\{ \, \sum_{i=1}^m \delta^-_i \middle|\,
                  (\bm{x}-\bm{\delta}^-\otimes \bm{x},\bm{y})\in \partial^s(\Tvrs),\, \bm{\delta}^- \geq \bm{0} 
          \right\} \label{eq:last.Dminus} \\
						&\ge 1- (1-\theta^*)=\theta^*>0, \notag          
	\end{align}
	where the last inequality follows from~\eqref{eq:target2}.
	Here, let $D^-(\bm{x},\bm{y};\Tvrs)$ denote the minimum value in~\eqref{eq:last.Dminus}.
	This is the minimum input-oriented weighted $\ell_1$-distance
from $(\bm{x},\bm{y})$ to  $\partial^s(\Tvrs)$.
	Hence, the target in~\eqref{eq:target2} attains the minimum input-oriented weighted $\ell_1$-distance
	$D^-(\bm x,\bm y;\Tvrs)=1-\theta^*$.

\par
For any   $ (\bm{x},\bm{y}) \in T_{\mbox{\tiny VRS-TO}}$,  there exists $r \in  I^+(\bm{y})$ such that  $(\bm{x},((\phi^*-1)\bm{e}_r+\bm{1}_s)\otimes\bm{y})\in \partial^s(\Tvrs)$ and we 
define the output-oriented least distance problem as follows:
\begin{align*}
&D^+(\bm{x},\bm{y};\Tvrs):= \min\left\{ 
\, \sum_{r=1}^s \delta^+_r \middle|
(\bm{x},\bm{y}+\bm{\delta}^+\otimes \bm{y})\in \partial^s(\Tvrs),\, \bm{\delta}^+ \geq \bm{0} 
\right\}.
\end{align*}
Then,   $(\phi^*-1)\bm{e}_r$ is a feasible solution to the output-oriented least distance  problem and
\begin{align*}
	\phi^*-1 \geq D^+(\bm{x},\bm{y};\Tvrs) &\geq  \min\left\{ 
\, \sum_{r=1}^s \delta^+_r \middle|\,
(\bm{x},\bm{y}+\bm{\delta}^+\otimes \bm{y})\in \partial^s(P),\, \bm{\delta}^+ \geq \bm{0} 
\right\}\\
&\geq  \min\left\{ 
\, \sum_{r\in I^+(\bm y)}  \delta^+_r \middle|\,
(\bm{x},\bm{y}+\bm{\delta}^+\otimes \bm{y})\in \partial^s(P),\, \bm{\delta}^+ \geq \bm{0} 
\right\}\\
&=  \min\left\{ 
\, \sum_{r\in I^+(\bm y)}  \delta^+_r \middle|\,
(\bm{x},\bm{y}+\bm{\delta}^+\otimes \bm{y})\in \partial^w(P),\, \bm{\delta}^+ \geq \bm{0} 
\right\}\\
&=\min_{r\in I^+(\bm y)}  \max\left\{ \delta^+_r \middle| \,
(\bm{x},\bm{y}+\delta^+_ry_r\bm{e}_r)\in P, \, \delta^+_r  \geq 0 
\right\}\\
&= \min_{r\in I^+(\bm y)}  \max\left\{ \phi_r \middle| \,
(\bm{x},\bm{\phi}\otimes \bm{y})\in P,  \, \phi_q=1\ \forall q \not=r 
\right\}-1=\phi^*-1.
\end{align*}

  In the case 
	$ \hat{F}(\bm{x},\bm{y};T_{\mbox{\tiny VRS-TO}})=   \frac{m+s-1+{1/\phi^*}}{m+s}$, we can show in the same manner that 
	$ \hat{F}(\bm{x},\bm{y};T_{\mbox{\tiny VRS-TO}})$ provides {a} closest target  from $(\bm{x},\bm{y})$ to $\partial^s(\Tvrs)$  based on  the output-oriented weighted $\ell_1$-distance $D^+$.} 
\par
The closest target also suggests a practical aspect that for (inefficient) DMUs, 
the extended max RGM provides targets that enable them to become
efficient with less effort.


\section{Frontier Assumption Check and Free Lunch Issue }
\label{sec:Comp}
We can see that the key assumption $\partial^s(P)=\partial^w(P)$ depends on
the choice of the trade-offs representation matrices $(R^-,R^+)$.
In this section, we demonstrate that checking the assumption
$\partial^s(P)=\partial^w(P)$ can be achieved by solving a series of LPs.
We also address the free lunch issue, and checking the occurrence of
the free lunch issue simply requires solving LPs.

\leavevmode\red{To this end, we first establish an auxiliary lemma.
An input--output vector $(\bm x,\bm y)\in P$ belongs to $\partial^s(P)$ 
if and only if the optimal value of the following linear programming problem
is zero:
\begin{align}
	\max & \quad \sum_{i=1}^m d^-_i + \sum_{r=1}^s d^+_r  \label{Add.obj} \\
	\text{s.t.} & \quad \sum_{j=1}^n \lambda_j \bm x_j + R^- \bm\pi = \bm x - \bm d^-, \label{Add.const.first} \\
							& \quad \sum_{j=1}^n \lambda_j \bm y_j + R^+ \bm\pi = \bm y + \bm d^+, \\
							& \quad \sum_{j=1}^n \lambda_j = 1, \\
							& \quad \bm\lambda\ge \bm 0,\ \bm\pi\ge \bm 0,\ \bm d^- \ge \bm0,\ \bm d^+\ge \bm 0. \label{Add.const.last}
\end{align}
An equivalent condition for $\partial^s(P)\neq \emptyset$ is shown as follows:
\begin{lemma}\label{lem:nonemptyness.frontier}
The condition $\partial^s(P)\ne\emptyset$ holds if and only if
there exists $(\bm{v},\bm{u})\in\bbR^{m+s}_{++}$ such that
\begin{align}
  \min\left\{
    \bm{v}\bm{r}^-_t-\bm{u}\bm{r}^+_t
    \,\middle|\, t=1,\ldots,K
  \right\}\geq 0.
  \label{eqConA}
\end{align}
\end{lemma}
\begin{proof}
\textbf{The ``only if'' part:}
Suppose that there exists $(\bm{x},\bm{y})\in\partial^s(P)$.
Then, the problem~\eqref{Add.obj}--\eqref{Add.const.last}
has optimal value zero.
By strong duality, its dual problem also has an optimal solution.
The dual problem is
\begin{align}
  \min &\quad \bm{v}\bm{x}-\bm{u}\bm{y}-\sigma \label{dAdd0}\\
  \text{s.t.} &\quad
    \bm{v}\bm{x}_j-\bm{u}\bm{y}_j\geq\sigma, \quad j=1,\ldots,n, \label{dAdd1}\\
  &\quad
    \bm{v}R^--\bm{u}R^+\geq\bm{0}, \label{dAdd2}\\
  &\quad
    \bm{v}\geq\bm{1}_m^\top,\ \bm{u}\geq\bm{1}_s^\top.
    \label{dAdd3}
\end{align}
An optimal solution $(\bm{v}^*,\bm{u}^*,\sigma^*)$
to~\eqref{dAdd0}--\eqref{dAdd3} therefore satisfies
$(\bm{v}^*,\bm{u}^*)\in\bbR^{m+s}_{++}$ and
\[
  \bm{v}^*\bm{r}^-_t-\bm{u}^*\bm{r}^+_t\geq 0,
  \qquad t=1,\ldots,K.
\]

\par
\textbf{The ``if'' part:}
Suppose that there exists
$(\bm{v},\bm{u})\in\bbR^{m+s}_{++}$ satisfying
$\bm{v}R^--\bm{u}R^+\geq\bm{0}$.
Let
\[
  \epsilon:=
  \min\left\{
    \min_{i=1,\ldots,m}v_i,\,
    \min_{r=1,\ldots,s}u_r
  \right\},
\]
and define $(\hat{\bm{v}},\hat{\bm{u}}) :=1/\epsilon (\bm{v},\bm{u})$ and
$\hat{\sigma}:= \min \left\{ \hat{\bm{v}}\bm{x}_j-\hat{\bm{u}}\bm{y}_j \mid j=1,\dots,n \right\}$.
Without loss of generality, we assume 
$\hat{\sigma} =\hat{\bm{v}}\bm{x}_1-\hat{\bm{u}}\bm{y}_1$.
For $(\bm{x},\bm{y})=(\bm{x}_1,\bm{y}_1)$,
$(\hat{\bm{v}},\hat{\bm{u}},\hat{\sigma})$ is a feasible solution to
the dual problem~\eqref{dAdd0}--\eqref{dAdd3},
and its objective function value is $\hat{\bm v}\bm x_1 - \hat{\bm u}\bm y_1 - \hat{\sigma} = 0$.
Hence, $(\hat{\bm v},\hat{\bm u},\hat{\sigma})$ is an optimal solution to
the dual problem~\eqref{dAdd0}--\eqref{dAdd3} for $(\bm x,\bm y) = (\bm x_1,\bm y_1)$.
Therefore, the optimal value of the primal problem~\eqref{Add.obj}--\eqref{Add.const.last}
for $(\bm x,\bm y)=(\bm x_1,\bm y_1)$ is also zero by the duality theorem.
This shows that $(\bm x_1,\bm y_1)\in\partial^s(P)$.
\end{proof}}

For the trade-offs representation matrices $(R^-,R^+)$,
a condition for $\partial^s(P)=\partial^w(P)$ is shown as follows:  
\begin{theorem}\label{Le4a.ineff}
\leavevmode\red{Suppose that $\partial^s(P)\neq \emptyset$.}
The set~$P$ given by~\eqref{eqP} satisfies  
\begin{align}
	P=\left\{(\bm{x},\bm{y})\left|
		\bm{v}\bm{x}-\bm{u}\bm{y}\geq \sigma\quad \forall (\bm{v},\bm{u},\sigma)\in W
\right\}\right.,\label{eqPW}
\end{align}
where 
\[
	W:=\left\{ (\bm{v},\bm{u},\sigma) \left| \begin{array}{l}
			\bm{v}\bm{1}_m + \bm{u}\bm{1}_s = 1,\\
			\bm{v}\bm{x}_j -\bm{u}\bm{y}_j -\sigma \geq 0,\ j=1,\ldots,n, \\
			\bm{v} R^--\bm{u}R^+ \geq   \bm{0},\ (\bm{v},\bm{u})\geq (\bm{0},\bm{0})
		\end{array}
\right\}\right..
\] 
For each $i=1,\dots,m$, let $v_i^*$ be the optimal value of
\begin{align}
	\min\left\{ v_i  \left| 
		\begin{array}{l}
			\bm{v}\bm{1}_m + \bm{u}\bm{1}_s = 1, \\
			\bm{v} R^--\bm{u}R^+ \geq   \bm{0}, (\bm{v},\bm{u})\geq (\bm{0},\bm{0})
\end{array} \right\}\right.,\label{eqVC}
\end{align}
and for each $r=1,\dots,s$, let  $u_r^*$ be the optimal value of 
\begin{align}
	\min\left\{ u_r  \left|
		\begin{array}{l}
			\bm{v}\bm{1}_m + \bm{u}\bm{1}_s = 1, \\
			\bm{v} R^--\bm{u}R^+ \geq  \bm{0}, (\bm{v},\bm{u})\geq (\bm{0},\bm{0})
		\end{array}
\right\}\right..\label{eqUC}
\end{align}
If $v_i^*>0$ for all $i=1,\ldots,m$ and $u^*_r>0$ for all $r=1,\ldots,s$,
then
\begin{align}
	\partial^s(P) = \partial^w(P).\label{assumption}
\end{align}
\end{theorem}
\begin{proof}
\leavevmode\red{By $\partial^s(P)\neq\emptyset$, Lemma~\ref{lem:nonemptyness.frontier} ensures that there exists $(\bm v,\bm u)\in\bbR^{m+s}_{++}$
such that $\bm v R^- - \bm u R^+ \ge \bm 0$.
Letting $\sigma := \min_{j=1,\dots,n} (\bm v \bm x_j - \bm u \bm y_j)$,
we have $\frac{1}{\bm v \bm{1}_m + \bm u \bm{1}_s} (\bm v,\bm u,\sigma)\in W$, which implies that $W\neq\emptyset$.}

Now we show the validity of~\eqref{eqPW}.
It follows from the definitions of~\eqref{eqP} and  $W$ and the duality theorem of LP that $(\bm{x},\bm{y}) \in P$ if and only if
\begin{align*}
& 0\leq \max\left\{ \delta \left| 
\begin{array}{l}
\sum_{j=1}^n \lambda_j \bm{x}_j + R^-\bm{\pi} \leq \bm{x} -\delta \bm{1}_m, \\
\sum_{j=1}^n \lambda_j \bm{y}_j + R^+\bm{\pi} \geq \bm{y}+\delta \bm{1}_s,  \\
\sum_{j=1}^n \lambda_j =1,\, \bm{\lambda}\geq \bm{0}, \bm{\pi}\geq \bm{0}
\end{array}
\right\} \right.=\min\left\{ \bm{v}\bm{x}-\bm{u}\bm{y}-\sigma \left| (\bm{v},\bm{u},\sigma)\in W
\right\} \right.\\
\iff&  \bm{v}\bm{x}-\bm{u}\bm{y}-\sigma \geq 0 \mbox{  for all } (\bm{v},\bm{u},\sigma)\in W.
\end{align*}
This means that~\eqref{eqPW} is valid. 
\par
For all $l=1,\ldots,L$, $(\bm{v}^l,\bm{u}^l,\sigma^l)$ of~\eqref{P}   satisfies $\bm{v}^l\bm{1}_m+\bm{u}^l\bm{1}_s>0$.
Let $(\hat{\bm{v}}^l,\hat{\bm{u}}^l,\hat{\sigma}^l):=\frac{1}{\bm{v}^l\bm{1}_m+\bm{u}^l\bm{1}_s}(\bm{v}^l,\bm{u}^l,\sigma^l)$ 
and $(\bm{x}^l,\bm{y}^l)\in \arg\min\{ \bm{v}^l\bm{x}-\bm{u}^l\bm{y} | (\bm{x},\bm{y})\in P  \}$ for all $l=1,\ldots,L$.
Then, 
\[
	\hat{\sigma}^l=\frac{1}{\bm{v}^l\bm{1}_m+\bm{u}^l\bm{1}_s}\sigma^l=\frac{1}{\bm{v}^l\bm{1}_m+\bm{u}^l\bm{1}_s}(\bm{v}^l\bm{x}^l-\bm{u}^l\bm{y}^l)=\hat{\bm{v}}^l\bm{x}^l-\hat{\bm{u}}^l\bm{y}^l,
\]
and 
\begin{align*}
	0&\leq \frac{1}{\bm{v}^l\bm{1}_m+\bm{u}^l\bm{1}_s}( \bm{v}^l\bm{x}_j-\bm{u}^l\bm{y}_j -\sigma^l)= \hat{\bm{v}}^l\bm{x}_j-\hat{\bm{u}}^l\bm{y}_j -\hat{\sigma}^l,\qquad  j=1,\ldots,n,\\
0&\leq 
\frac{1}{\bm{v}^l\bm{1}_m+\bm{u}^l\bm{1}_s}( \bm{v}^l(\bm{x}^l+R^-\bm{e}_t)-\bm{u}^l(\bm{y}^l+R^+\bm{e}_t) -\sigma^l)=\hat{\bm{v}}^l(\bm{x}^l+R^-\bm{e}_t)-\hat{\bm{u}}^l(\bm{y}^l+R^+\bm{e}_t) -\hat{\sigma}^l \\
&= \hat{\bm{v}}^l\bm{x}^l- \hat{\bm{u}}^l\bm{y}^l -\hat{\sigma}^l+ \hat{\bm{v}}^lR^-\bm{e}_t - \hat{\bm{u}}^lR^+\bm{e}_t=
\hat{\bm{v}}^lR^-\bm{e}_t - \hat{\bm{u}}^lR^+\bm{e}_t, \qquad t=1,\ldots,K,
\end{align*} 
where $\bm{e}_t$ is a unit vector whose $t$th component is $1$. 
Since $\hat{\bm{v}}^l\bm{1}_m+\hat{\bm{u}}^l\bm{1}_s=1$ and 
$(\hat{\bm{v}}^l,\hat{\bm{u}}^l) \in \bar{\bbR}^{m+s}_+$, the above 
inequalities show that
$(\hat{\bm{v}}^l,\hat{\bm{u}}^l,\hat{\sigma}^l)\in W$ for all $l=1,\ldots,L$.
By the assumption of $v^*_i>0$ for all $i=1,\dots,m$ and $u^*_r>0$ for all $r=1,\dots,s$,
we have $(\hat{\bm{v}}^l,\hat{\bm{u}}^l) \in \bbR^{m+s}_{++}$ for all $l=1,\ldots,L$.
It follows from Lemma~\ref{lem:positive.vandu} that the equation~\eqref{assumption} holds. 
\end{proof}
Checking the sufficient condition for $\partial^s(P)=\partial^w(P)$
is easy since it is enough to solve two types of LPs~\eqref{eqVC}
and~\eqref{eqUC}.

The free lunch issue of the PPS  $T_{\mbox{\tiny VRS-TO}}$ is characterized as follows.
\begin{theorem}\label{FreelunchCheck}  
	\leavevmode\red{Suppose that $\partial^s(P)\neq\emptyset$.}
Consider $P$ defined by~\eqref{P}.
If the optimal value of 
\begin{align}
	\max \left.\left\{ \sum_{r=1}^s  d^+_r\, \right|\, (\bm{0},\bm{d}^+)\in P, \ \bm{d}^+\in \bbR^s_+ \right\}  \label{Free}
\end{align}
is positive, 
then there exists
$\bm{y}\in \bar{\bbR}^s_{+}$ satisfying
$(\bm{0},\bm{y})\in\Tvrs$ and vice versa. 
\end{theorem}
\begin{proof}
	\textbf{Sufficiency ($\Leftarrow$)}:
	\red{Lemma~\ref{lem:nonemptyness.frontier} implies that there exists $(\bm v^1,\bm u^1)\in\bbR^{m+s}_{++}$
	such that $\bm v^1 R^- - \bm u^1 R^+ \ge \bm 0$.
	Let $\sigma^1 := \min_{j=1,\dots,n} (\bm{v}^1\bm x_j - \bm{u}^1 \bm y_j)$ and $\epsilon := \min_{r=1,\dots,s} u^1_r > 0$.
	Then, $1/\epsilon (\bm v^1,\bm u^1,\sigma^1)$ is a feasible solution to
the following dual problem of~\eqref{Free}:}
 \begin{align}
  \min\left\{ -\sigma \left|
  \bm{v}\bm{x}_j - \bm{u}\bm{y}_j \geq \sigma  \ (j=1,\ldots,n),
	\ \bm{v}R^- - \bm{u}R^+ \ge \bm{0},  \bm{u} \geq \bm{1}_s^{\top}, \bm{v}\geq \bm{0}
\right.  \right\}. \label{DFree}
  \end{align}
  Since there exists $\bm{y}\in \bar{\bbR}^s_{+}$ satisfying $(\bm{0},\bm{y})\in\Tvrs\subseteq P$, 
	$\bm{y}$ is  a feasible solution to~\eqref{Free}, and its objective function value $\sum_{r=1}^s y_r$ is positive.
The feasibility of both~\eqref{Free} and~\eqref{DFree} guarantees that
each problem has an optimal solution by the duality theorem.
Therefore,~\eqref{Free} and~\eqref{DFree} have the same positive optimal value.
\par
\textbf{Necessity ($\Rightarrow$)}:
	Let $\bm{d}^{+*}$ be an optimal solution to~\eqref{Free}.
	Then, we have $\bm{d}^{+*}\in \bar{\bbR}^s_+$, and hence $(\bm{0},{\bm{d}}^{+*})\in P\cap( \bbR^{m}_+\times \bar{\bbR}^s_{+})= \Tvrs$. 
\end{proof}
The potential  free lunch issue in $\hat{F}$ can be exposed by the 
free lunch characterization of Theorem~\ref{FreelunchCheck},
which is to solve a single LP~\eqref{Free}.
We will reveal via numerical experiments in Section~\ref{sec:examples} that
the target with the existing model has the free lunch issue for many DMUs
in a real-world dataset, but our model does not.
\section{An Illustrative Example}\label{sec:examples}
We employ an actual dataset on the performance of
participating nations at the Paris 2024 Summer Olympic Games,
which includes zero outputs for many DMUs,
to illustrate how to implement production trade-offs into  a pair of  
matrices $(R^-,R^+)$ and to compare the practicality of $\hat{F}$ with $\barEfgl$.  
The actual dataset contains 90 observations  on three outputs (the numbers of gold, silver and bronze medals) and three inputs (GDP per capita, population, and the number of teams) of 
the participating nations that received at least one medal.
This dataset is listed in the  lexicographical order of the numbers of gold, silver, and bronze medals won by each nation, 
which is shown  in  columns 3 to 8 of Tables~\ref{tab:Paris1}--\ref{tab:Paris3}.
DMU${}_j$ denotes the $j$th nation, and $\bm{x}_j=(x_{1j},x_{2j},x_{3j})^\top$ denotes
GDP per capita, population, and the number of teams of DMU$_j$, respectively.
Similarly, $\bm{y}_j=(y_{1j},y_{2j},y_{3j})^\top$ denotes 
the numbers of gold, silver, and bronze medals won by DMU$_j$, respectively.

\begin{table}[phtb!]
\tabcolsep = 1pt
\centering
\caption{Paris 2024 Summer Olympic Games (DMU$_1,\ldots,\mathrm{DMU}_{30}$) }\label{tab:Paris1}
\begin{tabular}{p{16pt}p{58pt}p{10pt}*{2}{p{42pt}}p{12pt}p{5pt}*{3}{p{12pt}}p{2pt}p{32pt}p{56pt}p{2pt}p{32pt}p{36pt}}
\hline \hline \small
No.&DMU&&\multicolumn{3}{c}{Inputs}&&\multicolumn{3}{c}{Outputs}&&\multicolumn{2}{c}{$\hat{F}$}&&\multicolumn{2}{c}{$\Efgl$}\\ \cline{12-13} \cline{15-16} \cline{4-6} \cline{8-10}
&&&$x_1$&$x_2$&$x_3$&&$y_1$&$y_2$&$y_3$&&Score&Target&&Score&Target\\\hline
1&USA&&81632&334915&619&&40&44&42&&1.000&&&1.000&\\
2&China&&12514&1425671&398&&40&27&24&&0.883&$y_1^*=134.3$&&0.322&\\
3&Japan&&33806&123400&431&&20&12&13&&0.971&$y_1^*=24.1$&&0.840&\\
4&Australia&&65434&26704&475&&18&19&16&&1.000&&&1.000&\\
5&France&&46001&68413&600&&16&26&22&&0.995&$y_1^*=16.5$&&0.982&\\
6&Netherlands&&62719&17564&289&&15&7&12&&0.982&$y_1^*=16.8$&&0.818&\\
7&Great Britain&&49099&67510&342&&14&22&29&&1.000&&&1.000&\\
8&Korea&&33192&51784&147&&13&9&10&&0.989&$y_1^*=13.9$&&0.938&\\
9&Italy&&38326&58862&397&&12&13&15&&0.947&$y_1^*=17.6$&&0.811&\\
10&Germany&&52727&83240&457&&12&13&8&&0.916&$y_1^*=24.2$&&0.651&\\
11&NewZealand&&47537&5228.1&208&&10&7&3&&0.950&$y_1^*=14.3$&&0.587&\\
12&Canada&&53548&40454&332&&9&7&11&&0.913&$y_1^*=18.8$&&0.600&\\
13&Uzbekistan&&2523&36296&88&&8&2&3&&0.961&$y_1^*=10.5$&&0.538&\\
14&Hungary&&22147&9635&177&&6&7&6&&0.929&$y_1^*=10.4$&&0.596&\\
15&Spain&&33071&47433&401&&5&4&9&&0.877&$y_1^*=19.1$&&0.446&\\
16&Sweden&&56225&10420&125&&4&4&3&&0.878&$y_1^*=14.8$&&0.317&\\
17&Kenya&&2113&55610&74&&4&2&5&&0.891&$y_1^*=11.5$&&0.334&\\
18&Norway&&87739&5406&109&&4&1&3&&0.872&$y_1^*=17.3$&&0.233&\\
\leavevmode\red{19}&Ireland&&104272&5211&143&&4&0&3&&0.868&$y_1^*=19.2$&&0.213&\\
\leavevmode\red{20}&Brazil&&10642&203080&290&&3&7&10&&0.852&$y_2^*=63.9$&&0.286&\\
\leavevmode\red{21}&Iran&&4663&89137&41&&3&6&3&&0.869&$y_1^*=13.9$&&0.299&\\
\leavevmode\red{22}&Ukraine&&5337&36515&141&&3&5&4&&0.879&$y_1^*=10.9$&&0.354&\\
\leavevmode\red{23}&Romania&&18176&19060&96&&3&4&2&&0.880&$y_1^*=10.8$&&0.286&\\
\leavevmode\red{24}&Georgia&&8173&3727&28&&3&3&1&&0.901&$y_1^*=7.4$&&0.234&$\bm{x}^*=\bm{0}$\\
\leavevmode\red{25}&Belgium&&53659&11630&177&&3&1&6&&0.865&$y_1^*=15.6$&&0.269&\\
\leavevmode\red{26}&Bulgaria&&15854&6640&46&&3&1&3&&0.889&$y_1^*=9$&&0.234&$\bm{x}^*=\bm{0}$\\
\leavevmode\red{27}&Serbia&&11327&6716&114&&3&1&1&&0.883&$y_1^*=10$&&0.184&$\bm{x}^*=\bm{0}$\\
28&Czechia&&30600&10700&111&&3&0&2&&0.873&$y_1^*=12.7$&&0.178&$\bm{x}^*=\bm{0}$\\
29&Denmark&&68300&5920&131&&2&2&5&&0.854&$y_1^*=15.9$&&0.228&\\
30&Azerbaijan&&7525&10470&48&&2&2&3&&0.875&$y_1^*=8.1$&&0.215&\\
 \hline \hline
\end{tabular}
\end{table}
\begin{table}[phtb]
\tabcolsep = 1pt
\centering
\caption{   Paris 2024 Summer  Olympic Games   (DMU$_{31}, \ldots, $DMU$_{60}$) }\label{tab:Paris2}
\begin{tabular}{p{16pt}p{58pt}p{10pt}*{2}{p{36pt}}p{12pt}p{5pt}*{3}{p{12pt}}p{2pt}p{32pt}p{50pt}p{2pt}p{32pt}p{36pt}}
\hline \hline \small
No.&DMU&&\multicolumn{3}{c}{Inputs}&&\multicolumn{3}{c}{Outputs}&&\multicolumn{2}{c}{$\hat{F}$}&&\multicolumn{2}{c}{$\Efgl$}\\ \cline{12-13} \cline{15-16} \cline{4-6} \cline{8-10}
&&&$x_1$&$x_2$&$x_3$&&$y_1$&$y_2$&$y_3$&&Score&Target&&Score&Target\\\hline
31&Croatia&&21347&3850&73&&2&2&3&&0.868&$y_1^*=9.7$&&0.215&$\bm{x}^*=\bm{0}$\\
32&Cuba&&11194&11147&61&&2&1&6&&0.873&$y_1^*=8.4$&&0.236&\\
33&Bahrain&&28262&1517&14&&2&1&1&&0.867&$y_1^*=9.9$&&0.143&$\bm{x}^*=\bm{0}$\\
34&Slovenia&&32233&2100&95&&2&1&0&&0.861&$y_1^*=12.1$&&0.111&$\bm{x}^*=\bm{0}$\\
35&Chinese Taipei&&32444&23500&60&&2&0&5&&0.860&$y_1^*=12.6$&&0.188&\\
36&Austria&&57081&9090&84&&2&0&3&&0.856&$y_1^*=14.9$&&0.154&$\bm{x}^*=\bm{0}$\\
37&HongKong&&50030&7474&34&&2&0&2&&0.859&$y_1^*=13.2$&&0.134&$\bm{x}^*=\bm{0}$\\
38&Philippines&&3868&117300&22&&2&0&2&&0.852&$y_1^*=17.5$&&0.133&\\
39&Algeria&&5324&45300&46&&2&0&1&&0.862&$y_1^*=11.8$&&0.110&\\
40&Indonesia&&4942&277500&29&&2&0&1&&0.844&$y_1^*=32.6$&&0.098&\\
41&Israel&&52219&9600&89&&1&5&1&&0.856&$y_2^*=37.1$&&0.166&\\
42&Poland&&21996&37470&226&&1&4&5&&0.852&$y_2^*=36$&&0.229&$\bm{x}^*=\bm{0}$\\
43&Kazakhstan&&13117&19800&79&&1&3&3&&0.853&$y_2^*=26$&&0.177&$\bm{x}^*=\bm{0}$\\
44&Jamaica&&6876&2820&65&&1&3&2&&0.858&$y_2^*=20.4$&&0.158&$\bm{x}^*=\bm{0}$\\
45&South Africa&&6138&60000&141&&1&3&2&&0.848&$y_2^*=34.4$&&0.158&$\bm{x}^*=\bm{0}$\\
46&Thailand&&7337&71600&52&&1&3&2&&0.848&$y_2^*=34.8$&&0.158&\\
47&Ethiopia&&1511&126500&33&&1&3&0&&0.845&$y_2^*=44.7$&&0.106&\\
48&Switzerland&&100413&8770&137&&1&2&5&&0.846&$y_3^*=67.8$&&0.168&\\
49&Ecuador&&6582&17900&40&&1&2&2&&0.852&$y_1^*=8.7$&&0.139&\\
50&Portugal&&27880&10300&75&&1&2&1&&0.848&$y_1^*=11.5$&&0.117&\\
51&Greece&&22805&10400&100&&1&1&6&&0.855&$y_3^*=45.6$&&0.176&$\bm{x}^*=\bm{0}$\\
52&Argentina&&14024&45700&143&&1&1&1&&0.845&$y_1^*=14.4$&&0.097&$\bm{x}^*=\bm{0}$\\
53&Egypt&&3728&112700&157&&1&1&1&&0.842&$y_1^*=19.4$&&0.097&\\
54&Tunisia&&4192&12500&26&&1&1&1&&0.854&$y_1^*=8.1$&&0.097&$\bm{x}^*=\bm{0}$\\
55&Botswana&&7642&2640&14&&1&1&0&&0.855&$y_1^*=7.6$&&0.068&$\bm{x}^*=\bm{0}$\\
56&Chile&&16816&19600&48&&1&1&0&&0.849&$y_1^*=10.9$&&0.068&\\
57&SaintLucia&&13546&180&4&&1&1&0&&0.854&$y_1^*=8$&&0.068&$\bm{x}^*=\bm{0}$\\
58&Uganda&&1139&48500&25&&1&1&0&&0.848&$y_1^*=11.2$&&0.068&\\
59&Dominican Republic&&11187&11300&59&&1&0&2&&0.851&$y_1^*=9.4$&&0.086&$\bm{x}^*=\bm{0}$\\
60&Guatemala&&5369&18200&16&&1&0&1&&0.852&$y_1^*=8.8$&&0.067&$\bm{x}^*=\bm{0}$\\
 \hline \hline
\end{tabular}
\end{table}
\begin{table}[phtb]
\tabcolsep = 1pt
\centering
\caption{  Paris 2024 Summer  Olympic Games  (DMU$_{61}, \ldots, $DMU$_{90}$) }\label{tab:Paris3}
\begin{tabular}{p{16pt}p{62pt}p{10pt}*{2}{p{42pt}}p{12pt}p{5pt}*{3}{p{12pt}}p{2pt}p{32pt}p{55pt}p{2pt}p{32pt}p{36pt}}
\hline \hline \small
No.&DMU&&\multicolumn{3}{c}{Inputs}&&\multicolumn{3}{c}{Outputs}&&\multicolumn{2}{c}{$\hat{F}$}&&\multicolumn{2}{c}{$\Efgl$}\\ \cline{12-13} \cline{15-16} \cline{4-6} \cline{8-10}
&&&$x_1$&$x_2$&$x_3$&&$y_1$&$y_2$&$y_3$&&Score&Target&&Score&Target\\\hline
61&Morocco&&3889&37800&61&&1&0&1&&0.848&$y_1^*=11.2$&&0.067&$\bm{x}^*=\bm{0}$\\
62&Dominica&&8837&73&4&&1&0&0&&0.855&$y_1^*=7.6$&&0.039&$\bm{x}^*=\bm{0}$\\
63&Pakistan&&1461&240500&7&&1&0&0&&0.839&$y_1^*=28.6$&&0.039&\\
64&Turkey&&12849&85900&101&&0&3&5&&0.846&$y_3^*=65.3$&&0.106&\\
65&Mexico&&13642&129420&108&&0&3&2&&0.843&$y_2^*=52.4$&&0.070&\\
66&Armenia&&8153&2800&15&&0&3&1&&0.856&$y_2^*=21.6$&&0.058&$\bm{x}^*=\bm{0}$\\
67&Colombia&&6972&51900&88&&0&3&1&&0.848&$y_2^*=33.7$&&0.058&$\bm{x}^*=\bm{0}$\\
68&Kyrgyzstan&&1843&6700&79&&0&2&4&&0.852&$y_3^*=35.3$&&0.078&$\bm{x}^*=\bm{0}$\\
69&North\,Korea&&599&26100&14&&0&2&4&&0.850&$y_3^*=38.7$&&0.078&$\bm{x}^*=\bm{0}$\\
70&Lithuania&&27026&2800&51&&0&2&2&&0.845&$y_2^*=28.2$&&0.055&$\bm{x}^*=\bm{0}$\\
71&India&&2500&1438069&112&&0&1&5&&0.835&$y_3^*=468.8$&&0.048&\\
72&Moldova&&6832&2600&26&&0&1&3&&0.848&$y_3^*=35.2$&&0.050&$\bm{x}^*=\bm{0}$\\
73&Kosovo&&5917&1800&9&&0&1&1&&0.841&$y_2^*=20.5$&&0.027&$\bm{x}^*=\bm{0}$\\
74&Cyprus&&34957&1250&15&&0&1&0&&0.839&$y_2^*=30.1$&&0.014&$\bm{x}^*=\bm{0}$\\
75&Fiji&&5993&940&36&&0&1&0&&0.841&$y_2^*=21.5$&&0.014&$\bm{x}^*=\bm{0}$\\
76&Jordan&&4498&11300&12&&0&1&0&&0.841&$y_2^*=22.5$&&0.014&\\
77&Mongolia&&5668&3400&32&&0&1&0&&0.841&$y_2^*=21.8$&&0.014&$\bm{x}^*=\bm{0}$\\
78&Panama&&18726&4400&8&&0&1&0&&0.840&$y_2^*=25.4$&&0.014&$\bm{x}^*=\bm{0}$\\
79&Tajikistan&&1184&10300&14&&0&0&3&&0.847&$y_3^*=35.6$&&0.025&$\bm{x}^*=\bm{0}$\\
80&Albania&&7957&2800&8&&0&0&2&&0.843&$y_3^*=36.1$&&0.017&$\bm{x}^*=\bm{0}$\\
81&Grenada&&11624&125&6&&0&0&2&&0.842&$y_3^*=36.6$&&0.017&$\bm{x}^*=\bm{0}$\\
82&Malaysia&&12570&34100&26&&0&0&2&&0.840&$y_3^*=48.3$&&0.017&$\bm{x}^*=\bm{0}$\\
83&PuertoRico&&36369&3300&51&&0&0&2&&0.840&$y_3^*=49.5$&&0.017&$\bm{x}^*=\bm{0}$\\
84&C\^ote~d'Ivoire&&2572&29000&13&&0&0&1&&0.837&$y_3^*=42.3$&&0.008&$\bm{x}^*=\bm{0}$\\
85&CaboVerde&&4368&600&7&&0&0&1&&0.838&$y_3^*=33.1$&&0.008&$\bm{x}^*=\bm{0}$\\
86&Peru&&7933&34000&26&&0&0&1&&0.837&$y_3^*=46.6$&&0.008&\\
87&Qatar&&78696&2800&13&&0&0&1&&0.836&$y_3^*=56.8$&&0.008&$\bm{x}^*=\bm{0}$\\
88&Singapore&&84734&5900&23&&0&0&1&&0.836&$y_3^*=59.9$&&0.008&$\bm{x}^*=\bm{0}$\\
89&Slovakia&&24337&5400&28&&0&0&1&&0.837&$y_3^*=44.3$&&0.008&$\bm{x}^*=\bm{0}$\\
90&Zambia&&1381&20500&31&&0&0&1&&0.837&$y_3^*=40.4$&&0.008&$\bm{x}^*=\bm{0}$\\
 \hline \hline
\end{tabular}
\end{table}
The VRS PPS generated by 90 DMUs' input--output vectors is given by
\[
  \left\{
  \left(\bm{x},\bm{y}\right)
  \begin{array}{|l}
    \sum_{j\in {\cal B}} \lambda_j \bm{x}_j \leq \bm{x},  \\
    \sum_{j\in {\cal B}} \lambda_j \bm{y}_j  \geq \bm{y}, \\
    \sum_{j\in {\cal B}} \lambda_j =1,  \lambda_j \geq 0 \ (j\in {\cal B})
  \end{array}
  \right\},
\]
where 
\[
  \begin{aligned}
  {\cal B}:=\label{loc:efficient-dmu-indices}
	\{1,2,3,4,5,6,7,8,11,13,14,17,\red{21},\red{22},\red{24},32,33,\\
  44,57,62,63,66,68,69,72,73,75,79,81,85\}
  \end{aligned}
\]
is the set of efficient DMUs evaluated by the BCC DEA model~\cite{banker1984some},
which is a standard DEA model.
Hence, the proportion of efficient DMUs by the BCC DEA model is
$|{\cal B}|/90=30/90=1/3$.
\par 
To reduce the proportion,
we put the following production trade-offs assumptions:
\begin{itemize}
	\item any nation can exchange one silver medal for one bronze medal while keeping its inputs constant;
\item any nation can move its input--output vector along directions of the change of the top nations towards other nations; 
\item  any nation can move  its input--output vector  along directions of the change of Great Britain (DMU$_7$) towards those top nations. 
\end{itemize}

The first assumption is implemented as~\eqref{00output}. 
The top nations of the second and third assumptions are chosen
from all 30 DMU$_j$ ($j \in {\cal B}$) as indicated in the following 
10 DMUs:
\[
{\cal B}_{11-}:=\{1,2,3,4,5,6,7,8,11,13,14\}\setminus \{2\}.
\] 
Let ${\cal B}_{11}:=\{1,2,3,4,5,6,7,8,11,13,14\}$.
Then, ${\cal B}_{11}$ is the top 11  DMUs in~$\cal B$.  

Based on this set, we first attempted to check the production trade-offs,
but there is no feasible solution to the following system:
\begin{align}
&\bm{v}\bigl(  \bm{x}_p -\bm{x}_q \bigr) - \bm{u}\bigl(  \bm{y}_p -\bm{y}_q \bigr) \geq 0\quad 
\forall (p,q) \in \bigl(\{1,\ldots,90\} \setminus {\cal B}_{11}\bigr) \times {\cal B}_{11},
\label{eq58} \\
&\sum_{i=1}^3 v_i + \sum_{r=1}^3 u_r =1, \label{eq59}\\
&\bm{v}\geq \bm{0}, \ \bm{u} \geq \bm{0}.\label{eq60}
\end{align}  


\leavevmode\red{We selected a subset of DMUs to exclude from ${\cal B}_{11}$ to make~\eqref{eq58}--\eqref{eq60} feasible.
Our selection policy was to keep both the number of excluded DMUs and the optimal value
of~\eqref{Free} as small as possible.
Therefore, we excluded DMU$_2$ from ${\cal B}_{11}$.
With the resulting reference set ${\cal B}_{11-}$, the following system is feasible together with~\eqref{eq59} and~\eqref{eq60}:}

\begin{align*}
&\bm{v}\bigl(  \bm{x}_p -\bm{x}_q \bigr) - \bm{u}\bigl(  \bm{y}_p -\bm{y}_q \bigr) \geq 0\ \forall (p,q) \in \bigl(\{1,\ldots,90\} \setminus {\cal B}_{11-}\bigr) \times {\cal B}_{11-},\\
&\bm{v}\bigl(  \bm{x}_p -\bm{x}_q \bigr) - \bm{u}\bigl(  \bm{y}_p -\bm{y}_q \bigr) \geq 0\ \forall (p,q) \in \bigl({\cal B}_{11-}\setminus\{7\}\bigr) \times \{7\},\\
& u_2 - u_3 \geq 0,
\end{align*}    
where DMU$_7$ represents Great Britain.

We generate $R^-$ and $R^+$ defined by~\eqref{R+def} and~\eqref{00output}  
and define a VRS PPS with  production trade-offs by $R^-$ and $R^+$  as follows:  
\begin{align}
T_{\mbox{\tiny VRS-TO}}=
\left\{ 
\left(\bm{x},\bm{y}\right) 
\begin{array}{|l}
 \sum_{j=1}^{90} \lambda_j \bm{x}_j +\sum_{(p,q)\in K} \pi_{(p,q)} \bm{r}^-_{(p,q)} \leq \bm{x}, \\
 \sum_{j=1}^{90} \lambda_j \bm{y}_j +\sum_{(p,q)\in K} \pi_{(p,q)} \bm{r}^+_{(p,q)} \geq \bm{y}, \\
 \sum_{j=1}^{90} \lambda_j =1, \\
 \lambda_j \geq 0\ (j=1,\ldots,90), \\
  \pi_{(p,q)} \geq 0, \ (p,q)\in K
\end{array}
\right\}\cap (\bbR^{m}_+\times \bar{\bbR}^s_+),\label{PPS_olympic}
\end{align}
where
\begin{align}
  K:= 
  \Bigl(\bigl(\{1,\ldots,90\}\setminus {\cal B}_{11-}\bigr)
  \times {\cal B}_{11-}\Bigr)
  \cup
  \Bigl(\bigl(({\cal B}_{11-}\setminus\{7\}) \times \{ 7 \}\bigr)
  \cup \bigl( \{0\}\times\{0\} \bigr)\Bigr),
  \label{Kdef}
\end{align}
and
\begin{gather}
  \bm{r}^-_{(p,q)} := \bm{x}_p -\bm{x}_q,\ \bm{r}^+_{(p,q)} := \bm{y}_p -\bm{y}_q, \ (p,q)\in K\setminus\{(0,0)\},\label{R+def}                                                    \\
  \bm{r}^-_{(0,0)} := \bm{0},  \ \bm{r}^+_{(0,0)} :=(0,-1,1)^{\top}\red{,} \label{00output}
\end{gather}
\leavevmode\red{where~\eqref{00output} means that with current results, it is possible to increase the number of bronze medals by one,
even if the number of silver medals decreases by one.}

Since the optimal values 
$v_i^*$ of~\eqref{eqVC} for each $i=1,2,3$ and 
$u_r^*$ of~\eqref{eqUC} for each $r=1,2,3$ are positive,
as shown in Table~\ref{tabv*u*},
the PPS $T_{\mbox{\tiny VRS-TO}}$ defined by~\eqref{PPS_olympic} 
satisfies $\partial^s(P)=\partial^w(P)$ by Theorem~\ref{Le4a.ineff}.
\begin{table}[htb]
\centering 
\caption{   Optimal values of~\eqref{eqVC} and~\eqref{eqUC}  }\label{tabv*u*}
\begin{tabular}{*{3}{r}c*{3}{r}} \hline \hline
\multicolumn{3}{c}{Optimal values of  \eqref{eqVC}} & & \multicolumn{3}{c}{Optimal values of  \eqref{eqUC}} \\ \cline{1-3} \cline{5-7}
$v_1^*$ & $v_2^*$ & $v_3^*$ && $u_1^*$ & $u_2^*$ & $u_3^*$ \\ \hline 
0.00003 & 0.00006 & 0.00777 && 0.57774 & 0.14589&0.10255 \\ \hline \hline 
\end{tabular}
\end{table}

\par
The optimal values $u_r^*$, $r=1,2,3$, of~\eqref{eqUC}  satisfy
$u_1^*>u_2^*>u_3^*>0$, which is consistent with the natural relationship
between the values of gold, silver, and bronze medals.  
Since  the optimal value of~\eqref{Free} is  29.84183, 
Theorem~\ref{FreelunchCheck} implies that  
the PPS $T_{\mbox{\tiny VRS-TO}}$ defined by~\eqref{PPS_olympic}  has $(\bm{0},\bm{y})\in T_{\mbox{\tiny VRS-TO}}$
for some $\bm{y}\in \bar{\bbR}^s_{+}$,
which means free lunch.
Therefore, 
the PPS $T_{\mbox{\tiny VRS-TO}}$  allows an  unnatural production process.
As shown in Tables~\ref{tab:Paris1}--\ref{tab:Paris3},
the modified RGM $\Efgl$ 
sometimes provides a zero input target $\bm{x}^*=\bm{0}$ and a positive output target. 
For example,  in the last column of Table~\ref{tab:Paris1}, we see that DMU$_{\red{24}}$ has a zero input target $\bm{x}^*=\bm{0}$.
Besides,  since the output target of  {DMU$_{\red{24}}$} is positive,
the input--output target of {DMU$_{\red{24}}$} means free lunch.
Thus, the modified RGM $\Efgl$ 
exposes free lunch of $T_{\mbox{\tiny VRS-TO}}$ by the input target of the assessed DMU.
Since ${\bm{y}_{\red{24}}}>\bm{0}$, the target given by 
$\barEfgl$ is the same as that given by $\Efgl$.
Therefore, both   $\barEfgl$ and $\Efgl$ cannot avoid a free lunch target;
meanwhile,  $\hat{F}$ can avoid it  as  ensured by Theorem~\ref{Theorem2}.
\par
\leavevmode\red{Consider a virtual situation in which the input vector of DMU$_{24}$ is reduced from ${\bm{x}_{24}}$ to $\epsilon \mathbf{1}_3$, with $0<\epsilon<28$, and its output {vector} $(3,3,1)^{{\top}}$ remains {unchanged}.\label{loc:virtual-input-reduction-start}
Then, the virtual input--output vector $(\epsilon,\epsilon,\epsilon;3,3,1)^\top$ belongs to $\Tvrs$ {by free disposability, since} the target of DMU$_{24}$ provided by $\Efgl$ {has zero inputs and outputs no smaller than $(3,3,1)^\top$}.
\par
In fact, let $(\bm{\theta}^*,\bm{\phi}^*)$ be an optimal solution to {the minimization problem defining $\Efgl$ for DMU$_{24}$}.
Then, we have $\bm{\theta}^*=\bm 0$ and $\bm{\phi}^*\otimes\bm{y}_{24}=(3.746,7.530,4.826)^\top$.
Moreover, since $(-\epsilon\bm{1}_3,\bm{y}_{24})\le{(\bm{0},\bm{\phi}^*\otimes\bm{y}_{24})}$, $(\bm{\theta}^*,\bm{\phi}^*)$
is a feasible solution to $\Efgl(\epsilon\bm{1}_3,\bm{y}_{24};\Tvrs)$, and 
\[
	\bar{f}(\bm{\theta}^*,\bm{\phi}^*)\ge \Efgl(\epsilon\bm{1}_3,\bm{y}_{24};\Tvrs)\ge
	\Efgl(\bm{x}_{24},\bm{y}_{24};\Tvrs) = \bar{f}(\bm{\theta}^*,\bm{\phi}^*).
\]
Therefore, for any $\epsilon$ such that $0<\epsilon<28$, we have 
\[
	\Efgl(\epsilon\bm{1}_3,\bm{y}_{24};\Tvrs)= \Efgl(\bm{x}_{24},\bm{y}_{24};\Tvrs).
\]
Thus, reducing all inputs while keeping the outputs unchanged does not increase the efficiency score,
which violates (SM).\label{loc:virtual-input-reduction-end}}

\par
The free lunch target of DMU$_{54}$ implies violation of strong monotonicity (SM) of efficiency measures. 
In fact,  DMU$_{52}$ has the same efficiency score of $\Efgl$ as that of DMU$_{54}$ while 
$\bm{x}_{52}=(14024,45700,143)^{\top}> (4192,12500,26)^{\top}=\bm{x}_{54}$ and $\bm{y}_{52}=(1,1,1)^{\top}=\bm{y}_{54}$. 
Hence, $\Efgl(\bm{x}_{54},\bm{y}_{54})=  \Efgl(\bm{x}_{52},\bm{y}_{52})$ when 
$(\bm{x}_{52},-\bm{y}_{52})\geq (\bm{x}_{54},-\bm{y}_{54})$ and $(\bm{x}_{52},-\bm{y}_{52})\not= (\bm{x}_{54},-\bm{y}_{54})$. 
The efficiency score of DMU$_{54}$ by $\Efgl$  shows a counterexample to (SM) of both $\Efgl$ and $\barEfgl$.
Moreover, in Tables~\ref{tab:Paris1}--\ref{tab:Paris3},
we have the same  counterexamples to (SM) of both $\Efgl$ and $\barEfgl$ for the 11 DMUs indexed by $j\in \{\red{24},\red{26},\red{27},28,31,33,42,43,44,45,51\}$.\label{loc:sm-counterexample-dmu-indices} 
 \par
 On the other hand, by Theorem~\ref{Theorem1}  and $\bm{y}_j\in \bar{\bbR}^3_+$  for all $j=1,\ldots,90$,
	 $\hat{F}$ satisfies (SM); that is, for each $j=1,\ldots,90$, we have
\[
  \hat{F}(\bm{x}_j,\bm{y}_j; \Tvrs)
  >\hat{F}(\bm{x}',\bm{y}'; \Tvrs)
\]
 if $(\bm{x}_j,-\bm{y}_j)\leq (\bm{x}',-\bm{y}')$ and  $(\bm{x}_j,-\bm{y}_j)\not= (\bm{x}',-\bm{y}')$.
  
Let us conclude this section by summarizing the results of the 
experiment with a real-world dataset.
Although the considered PPS $\Tvrs$ with the dataset 
allows free lunch as shown in Theorem~\ref{FreelunchCheck}, 
the extended max RGM $\hat{F}$ always finds a nonzero input target, which 
is ensured by Theorem~\ref{Theorem2}.
This indicates that our model provides realistic targets,
while the existing model provides unrealistic targets with
zero inputs.
Furthermore, as shown by comparing DMU$_{52}$ and DMU$_{54}$,
the strong monotonicity is not satisfied for $\Efgl$.
For these reasons, our approach provides practical analysis results
on this dataset.

\section{Conclusions}\label{sec:conclusion}

In this paper, to tackle the boundary problem with the RGM,
we employed the closer target setting approach to the extended PPS.
This approach allows its efficiency measure to satisfy some desirable
properties such as indication and strong monotonicity.
The desirable properties of efficiency measures are valid for the extended PPS  under various types of returns to scale. 
\par
This approach provides an efficient and no free lunch target.
Moreover,  the target achieves the input-oriented or output-oriented  least distance  to the strongly efficient frontier,
and it requires only a single input or output to be improved.
Hence, the target may be reached from the assessed DMU  with less effort. 
\par
The computation of the efficiency measure only requires
solving a series of LPs while existing modifications for the RGM require
solving nonlinear programming problems.
Through a numerical experiment with a real-world dataset,
we illustrated how to choose  $(R^-,R^+)$ and
also verified the validity of the targets compared with $\Efgl$.
For given trade-offs representation matrices $(R^-,R^+)$,
we demonstrated that it is easy to verify the assumption
$\partial^s(P)=\partial^w(P)$ by solving LPs.
\par
The proposed efficiency measure has a lower bound $1-{1}/{(m+s)}$ from Lemma~\ref{lemma1};
that is, the range of efficiency scores becomes narrower as the number 
of inputs and outputs ($m+s$) increases.
We can avoid the narrow range of  efficiency scores by changing the objective function 
$\frac{1}{m+s} \left( \sum_{i=1}^m \theta_i + \sum_{r=1}^s {1}/{\phi_r} \right)$  to 
$\sum_{i=1}^m \theta_i + \sum_{r=1}^s {1}/{\phi_r} -(m+s-1)$, whose lower bound is $0$\red{,}
\red{and its efficiency score is\label{loc:rescaled-score-start} 
\[
	\tilde F(\bm x,\bm y;\Tvrs) = \max\left\{ \theta^*, 1/\phi^* \right\},
\]
where $\theta^*$ and $\phi^*$ are defined in~\eqref{theta.star} and~\eqref{phi.star}, respectively.\label{loc:rescaled-score-end}}
\par
As future work, the proposed approach for the RGM can be extended
to other representative non-radial DEA models such as the SBM DEA model~\cite{emrouznejad2025development,tone2001slacks}.
The SBM is a ratio form of input efficiency and output efficiency~\cite{fukuyama2014distance} while the objective function of the RGM is an additive form.  
Comparative studies of empirical benchmarking using non-radial DEA models, e.g.,  RGM, the proposed efficiency measure, and  SBM, are attractive.



\end{document}